%% file: main-revised.tex
\documentclass[submission,copyright,creativecommons]{eptcs}
\providecommand{\event}{ACT 2021} 

\newcommand{\omitthis}[1]{}
\newcommand{\mytitle}{A Categorical Semantics of Fuzzy Concepts in Conceptual Spaces}
\title{\mytitle} 

\usepackage{doi}

\author{Sean Tull\thanks{ 
We thank Bob Coecke, Steve Clark, Vincent Wang, Dimitri Kartsaklis and Sara Sabrina Zemljic for interesting discussions, and anonymous reviewers for ACT 2021 for helpful suggestions.}
\institute{Cambridge Quantum Computing}
\email{sean.tull@cambridgequantum.com} 
}
\def\titlerunning{\mytitle} 
\def\authorrunning{S. Tull}

\input{preamble}

\usepackage{enumitem} 

\newcommand{\CCon}{\CC} 
\newcommand{\CC}{\mathsf{Con}} 
\newcommand{\FCon}{\mathsf{FCon}} 

\newcommand{\Gauss}{\cat{Gauss}}

\newcommand{\QCon}{\mathsf{QCon}}

\newcommand{\Tr}{\mathsf{Tr}}

\begin{document}
\maketitle

\begin{abstract}
We define a symmetric monoidal category modelling fuzzy concepts and fuzzy conceptual reasoning within \Gardenfors{}' framework of conceptual (convex) spaces. We propose log-concave functions as models of fuzzy concepts, showing that these are the most general choice which both are well-behaved compositionally and satisfy \Gardenfors{}' criterion of quasi-concavity. We then generalise these to define the category of log-concave probabilistic channels between convex spaces, which allows one to model fuzzy reasoning with noisy inputs, and provides a novel example of a Markov category. 
\end{abstract}

\section{Introduction}

How can we model conceptual reasoning in a way which is formal and yet reflects the fluidity of concept use in human cognition? One answer to this question is given by Peter \Gardenfors{}' framework of \emph{conceptual spaces} \cite{gardenfors1993emergence,gardenfors2004conceptual,gardenfors2014geometry}, in which domains of conceptual reasoning are modelled by mathematical spaces and concepts are described geometrically, typically as \emph{convex} regions of these spaces. 

 The theory of conceptual spaces is defined only semi-formally, giving room for many authors to define their own mathematical formalisations  \cite{aisbett2001general,rickard2007reformulation,warglien2013semantics,lewis2016hierarchical,bechberger2017thorough}. A notable aspect of the framework is that it is \emph{compositional} in the sense that each overall conceptual space is given by composing various simpler \emph{domains} (e.g.~colour, sound, taste). This aspect makes the framework highly suited to formalisation in terms of \emph{monoidal categories}.
 
  Bolt et al.\ \cite{bolt2019interacting} have presented a categorical model of conceptual spaces within the \emph{DisCoCat} framework for natural language semantics \cite{coecke2010mathematical}, using the compact monoidal category $\ConvRel$ of convex relations. Here a conceptual space is modelled as a convex algebra $A$ and the meaning of a word (concept) as a convex subset. The Bolt et al.{} model demonstrates the use of monoidal categories in modelling the composition of conceptual spaces, and the correlations between domains contained within concepts.

  However, like most formalisations of conceptual spaces, the model of \cite{bolt2019interacting} is limited to describing only what we may call \emph{crisp} concepts, which are such that any point of the conceptual space either strictly is or is not a member, with no `grey areas'. In contrast, most discussions of concepts in the cognitive science literature acknowledge that concepts should be \emph{fuzzy} or \emph{graded} in the sense that for any point $x$ the degree of membership of a concept $C$ should form a scalar value $C(x) \in [0,1]$. For example, \Gardenfors{} suggests defining fuzzy membership based on distance from a central region \cite{gardenfors2014geometry} representing a \emph{prototype}  \cite{rosch1973natural}. 
  
  In this work we propose a mathematical definition of fuzzy concepts which is compositionally well-behaved and contains crisp concepts (convex regions) as a special case. Specifically we propose that fuzzy concepts on a space $X$ should be given by (measurable) \emph{log-concave} functions $C \colon X \to [0,1]$. We prove that these are essentially the largest class of functions which are closed compositionally and satisfy the criterion of \emph{quasi-concavity}, identified implicitly by \Gardenfors{} \cite[\S 2.8]{gardenfors2014geometry}, 
  which ensures that any point $z$ lying `in-between' two points $x, y$ belongs to the concept `as much' as they do. 
 
 Beyond concepts, a categorical approach is well-suited to describing \emph{processes} between spaces. To describe fuzzy processes between convex spaces mathematically, one typically works in the symmetric monoidal category $\ProbC$ whose 
 morphisms are probabilistic \emph{channels} $f \colon X \kto Y$ \cite{lawvere1962category,giry1982categorical,panangaden1998probabilistic}. These send each point $x$ of $X$ to a (sub-)probability measure (distribution) over $Y$. In this work, to model fuzzy \emph{conceptual} processes we introduce \emph{log-concave channels}, and prove that they form a symmetric monoidal subcategory $\LogCon$ of $\ProbC$. In particular, the \emph{effects} on a space $X$ in $\LogCon$ correspond precisely to the fuzzy concepts in our sense, while the \emph{states} of $X$ correspond to the widely studied class of \emph{log-concave probability measures} over the space $X$ \cite{saumard2014log,klartag2005geometry}. The latter include many standard distributions such as Gaussians, allowing us to model `noisy inputs' to our processes. More general morphisms $X \kto Y$ in $\LogCon$ may be seen as transformations of fuzzy concepts. 

   There are many avenues for further exploration of $\LogCon$ as a model of fuzzy conceptual processes, such as in the modelling of metaphors as maps between conceptual spaces, and in describing concepts formed by neural network systems with noisy inputs such as $\beta$-VAEs \cite{higgins2016beta}. More broadly, $\LogCon$ may be of wider use in categorical probability theory by providing a novel example of a \emph{Markov category} \cite{fritz2020representable}.

\paragraph{Related work} 
Our work extends the model of  Bolt et al.\  \cite{bolt2019interacting} to fuzzy concepts. Other such extensions include \cite{Wangfunctionals}, which considers arbitrary measurable functions into the interval $[-1,1]$, and \cite{coecke2018generalized}, which works with `generalised relations' rather than measure theory. Our definition of fuzzy concept is inspired by that of \Bechberger{} and \Kuhnberger{} \cite{bechberger2017thorough}, though they replace convexity by star-shapedness. 

\paragraph{Structure of article} 
We recall convex conceptual spaces and crisp concepts (Section \ref{sec:Conceptual-Spaces}), before proposing and justifying our definition of fuzzy concepts as log-concave functions (Section \ref{sec:fuzzy-concepts}). Next we recap categorical probability theory (Section \ref{sec:prob-cat}), before defining the category $\LogCon$ of log-concave channels as a model of conceptual processes (Section \ref{sec:log-concave-channels}). Our main results, Theorems \ref{thm:LogCon-anSMC} and \ref{thm:LogCon-Channels-Canonical},  prove that $\LogCon$ is the `largest' monoidal category whose effects are fuzzy concepts. We close by constructing examples of log-concave channels (Section \ref{sec:log-chan-examples}) and giving a toy example of conceptual reasoning (Section \ref{sec:toy-example}).

\section{Conceptual Spaces} \label{sec:Conceptual-Spaces}

Peter \Gardenfors'{} framework of \emph{conceptual spaces} provides an approach to the modelling of human and artificial conceptual reasoning, motivated by the cognitive sciences and mathematically based on the notion of `convexity' \cite{gardenfors2004conceptual,gardenfors2014geometry}. In this approach,  a conceptual space $C$ (such as that of images, foods or people) is described as a product of typically simpler spaces called `domains' (such as those of colours, sounds, tastes, temperatures \dots). Based on psychological experiments and arguments around learnability, \emph{concepts} are modelled as regions of a conceptual space which are \emph{convex}, meaning that any point lying in-between two instances of a concept is also an instance of that concept. In this article we work with an abstract definition of a conceptual space, without explicit reference to domains.

 We begin from the formalisation due to Bolt et al.~{}in terms of `convex algebras' \cite{bolt2019interacting}. Formally, these are algebras for the finite distribution monad. In detail, for any set $X$ we write $D(X)$ for the set of formal finite convex sums $\sum^n_{i=1} p_i |x_i\rangle$ of elements $x_i$ of $X$, where each $p_i \in [0,1]$ with $\sum^n_{i=1} p_i = 1$. These formal sums satisfy natural conditions suggested by the notation: for example, the order of the $p_i |x_i\rangle$ is irrelevant, and the sum is equal to $|x_i\rangle$ when $p_i = 1$.
 
\begin{definition}
A \emph{convex algebra} is a set $X$ coming with a function $\alpha \colon D(X) \to X$ satisfying 
\[
\alpha(|x\rangle) = x 
\qquad 
\alpha(\sum_i p_i \alpha(\sum_j q_{i,j} |x_{i,j} \rangle)) = 
\alpha(\sum_{i,j} p_i q_{i,j} |x_{i,j} \rangle)
\]
For any elements  $x_i \in X$ and positive weights $p_i$ with $\sum_i p_i = 1$ we may thus define a convex combination 
\begin{equation} \label{eq:conv-comb-notn} 
 \sum^n_{i=1} p_i x_i :=\alpha( \sum^n_{i=1} p_i |x_i\rangle) \in X 
\end{equation} 
We will denote binary convex combinations by 
\[
x\plusp y := px + (1-p)y
\]for $x,y \in X$ and $p \in [0,1]$. A map of convex algebras $f \colon X \to Y$ is called \emph{affine} when $f(\sum^n_{i=1} p_i x_i) = \sum^n_{i=1} p_i f(x_i)$ for all convex combinations. 
\end{definition} 

To discuss fuzzy notions later, we will require the tools of probability theory, and thus consider spaces which are \emph{measurable}. Recall that a \emph{measurable space} is a set $X$ with a $\sigma$-algebra $\Sigma_X \subseteq \pset(X)$, a family of subsets, which are called \emph{measurable}, which contains $X$ itself and is closed under complements and countable unions. A map of measurable spaces $f \colon X \to Y$ is  \emph{measurable} if $f^{-1}(M) \in \Sigma_X$ whenever $M \in \Sigma_Y$. Our basic model of a conceptual space is now the following. 

\begin{definition}
By a \emph{convex space} we mean a convex algebra $(X, \alpha)$ which is also a measurable space. A \emph{crisp concept} of $X$ is a measurable subset $C$ which is \emph{convex}, meaning that whenever $x_1, \dots, x_n \in C$ then $\sum^n_{i=1} p_i x_i \in C$ also. We denote the set of crisp concepts of $X$ by $\CC(X)$. 
\end{definition}

\begin{lemma} \label{lem:convex-comb-concepts} 
Let $X$ be a convex algebra. Then the convex subsets of $X$ themselves form a convex algebra via the Minkowski sum:
\begin{equation} \label{eq:Mink-sum}
A \plusp B := \{ a \plusp b \mid a \in A, b \in B\}
\end{equation} 
for $p \in [0,1]$. Hence if $X$ is a convex space such that $A \plusp B$ is measurable for all crisp concepts $A, B$, then $\CC(X)$ forms a convex algebra.
\end{lemma}

\begin{examples} \label{ex:conceptual-spaces}
Let us consider some examples of convex spaces and their crisp concepts; for more see \cite{bolt2019interacting}. 
\begin{enumerate} 
\item 
The unit interval $[0,1]$ forms a convex space with crisp concepts as sub-intervals. 
\item 
Any normed vector space $(X, \| - \|)$ forms a convex space via its Borel $\sigma$-algebra, which is generated by the open subsets. In particular $X = \mathbb{R}^n$ forms a convex space with either its Borel or Lebesgue $\sigma$-algebras. The crisp concepts are (measurable) convex subsets in the usual sense. 

\item 
Any convex measurable subset (crisp concept) $C$ of a convex space $X$ is again a convex space.
\item 
Any convex algebra $(X, \alpha)$ forms a convex space by using the \emph{discrete} $\sigma$-algebra $\Sigma_X = \pset(X)$. 

\item 
Any join semi-lattice $(X, \vee)$ forms a convex algebra (and hence space) by taking $x \plusp y = x \vee y$ for all $p \in (0,1)$ \cite{bolt2019interacting}, with crisp concepts as $\vee$-closed subsets. This allows one to consider discrete convex spaces, such as truth values $\{0,1\}$. 

\item \label{enum:product}
The \emph{product} of convex spaces $X, Y$ is the convex space on $X \times Y$ with operations 
\[
\sum^n_{i=1} p_i(x_i,y_i) = (\sum^n_{i=1}p_ix_i,\sum^n_{i=1}p_iy_i)
\]
for $x_i \in X$, $y_i \in Y$, and equipped with the \emph{product} $\sigma$-algebra $\Sigma_{X \times Y}$, the algebra generated by the subsets of the form  $A \times B$ for $A \in \Sigma_X$ and $B \in \Sigma_Y$. In particular when $C \in \CCon(X), D\in \CCon(Y)$ then $C \times D \in \CCon(X \times Y)$.
\item \label{enum:food-space-initial} 
In \cite{bolt2019interacting} toy conceptual spaces of colours and tastes are defined as follows. \emph{Colour space} is defined as the 3-dimensional cube 
\[
C=[0,1]^3 = \{(R,G,B) \mid  0 \leq R,G,B \leq 1\}
\]
of red-green-blue intensities. Specific points include (pure) green $g:= (0,1,0)$, yellow $ y:=(1,1,0)$, etc. We can define a crisp concept `green' for example as the (convex) open ball $G=B^\epsilon_g$ around green of a given radius $\epsilon > 0$, or more sharply as the singleton $\{g\}$. A simple \emph{taste space} $T$ is defined as the convex space (simplex) in $\mathbb{R}^4$ generated by the four points sweet, bitter, salt, and sour
\[
T = \{(t_1,t_2,t_3,t_4) \mid t_i \geq 0, \sum t_i = 1\}
\]
By taking the product of these convex spaces, we can form a toy \emph{food space} as 
\[
F = C \times T
\]
in which each food is modelled by a concept relating its colours and tastes. 
\item 
Any set of \emph{exemplar} points $E$ in a convex space define a convex set via their \emph{convex closure} $\overline{E}$, which is defined as the intersection of all convex subsets $C \subseteq X$ containing $E$, or equivalently as its set of convex combinations 
\[
\overline{E} = \{ \sum^n_{i=1} p_i e_i \mid e_i \in E\} 
\]
In spaces such as $\mathbb{R}^n$, the set $\overline{E}$ will be a closed crisp concept. We can think of $\overline{E}$ as a concept `learned' from these exemplars, with the convex closure allowing us to infer new instances of the concept. 
\end{enumerate} 
\end{examples}

\section{Fuzzy Concepts} \label{sec:fuzzy-concepts} 

The concepts described so far have been crisp, or `sharp', in that every element $x \in X$ either is or is not a member of the concept $C$, with either $x \in C$ or $x \notin C$. Real-life concept membership is arguably a more `fuzzy' notion, taking a value in the range $[0,1]$. For example, the concept `tall' can be seen as applying to a person in such a graded way (determined by their height). Fuzzy (continuous) concepts are also more easy to learn in neural networks, allowing for gradient descent, while crisp (discrete) ones require workarounds such as in \cite{maddison2016concrete}. We will take a concept to be a `fuzzy set', a map 
\[
C \colon X \to [0,1]
\]
where $C(x) \in [0,1]$ denotes the extent to which $x$ is an instance of the concept $C$. Such mappings are partially ordered, point-wise with $C \leq D$ whenever $C(x) \leq D(x) \forall x$.

Now fuzzy concepts should not be arbitrary mappings, but respect the convex structure of $X$ appropriately. \Gardenfors{} has suggested one structural feature that fuzzy concepts should satisfy, which amounts to the following requirement \cite[\S2.8]{gardenfors2014geometry}.\footnote{The analogous criterion for fuzzy \emph{star-shaped} sets is considered in \cite{bechberger2017thorough}.}

\begin{criterion} \label{crit:QC}
Let $X$ be a convex space. Fuzzy concepts $C \colon X \to [0,1]$ should be \emph{quasi-concave}, meaning that for all $x, y \in X$, $p \in [0,1]$ we have 
\[
C(x \plusp y) \geq \min \{C(x), C(y)\}
\]
Equivalently, each \emph{$t$-cut} $C_t := \{x \in X \mid C(x) \geq t\}$ should be a convex subset of $X$, for $t \in [0,1]$. 
\end{criterion}

This requirement is a natural one, stating that if $x$ and $y$ are both members of a concept to degree $t \in [0,1]$, then so is any point lying `between' them. Practically, it allows one to understand a fuzzy concept $C$ in terms of its `cuts' $\{C_t\}_{t \in [0,1]}$, ensuring that these will indeed form crisp concepts. 

 However, quasi-concavity is not fully sufficient if we wish to develop a compositional theory of fuzzy concepts, due to the following observation. For any $[0,1]$-valued maps $C$ on $X$ and $D$ on $Y$, we define
 \begin{align*} 
 C \otimes D \colon X \times Y &\to [0,1] \\ 
 (x,y) &\mapsto C(x)D(y)
 \end{align*}

\begin{remark}[Quasi-concavity is not compositional] For quasi-concave functions $C$ on $X$ and $D$ on $Y$, the map $C \otimes D$ is generally not quasi-concave. For example, take $X = Y = [0,1]$ with $C(x) = 1 - \frac{x}{2}$ and $D(y) = \frac{y^2 + 1}{2}$. Then $C \otimes D$ acts as $(0,0), (1,1) \mapsto \frac{1}{2} > 0.46875  \mapsfrom (\frac{1}{2},\frac{1}{2})$. 
\end{remark} 
 
 Hence we require a stricter definition to ensure that concepts may be composed. Luckily, there is a well-known class of quasi-concave functions which provide a well-behaved definition of fuzzy concept.

\begin{definition}[{Log-concavity/Fuzzy concepts}]  \label{def:fuzzy-concept} 
Let $X$ be a convex algebra. A function $f \colon X \to \mathbb{R}$ is \emph{log-concave} when for all $x, y \in X$ and $p \in [0,1]$ we have
\[
f(x \plusp y) \geq f(x)^pf(y)^{1-p}
\]
We define a \emph{fuzzy concept} on a convex space $X$ to be a measurable log-concave function $C \colon X \to [0,1]$. We denote the set of fuzzy concepts on $X$ by $\FCon(X)$.  
\end{definition} 

Any function $f$ which is \emph{concave}, with $f(x \plusp y) \geq f(x) \plusp f(y)$, is log-concave. Any log-concave function is quasi-concave. A function $f$ is log-concave iff 
\[
\log \circ f \colon X \to [-\infty, \infty]\] 
is concave, or equivalently iff
$f(x) = e^{u(x)}$
 with $u(x)$ concave. Log-concave functions on spaces  $\mathbb{R}^n$ form a well-studied class in statistics, including many standard functions from probability theory \cite{saumard2014log,klartag2005geometry}. They are known to be well-behaved under operations such as products, convolutions and marginalisation. 

Our definition of fuzzy concept is justified by the following result. Write $\QCon(X)$ for the set of quasi-concave functions $X \to [0,1]$.


\begin{theorem}[Log-concavity is canonical]  \label{thm:fuzzy-concepts-canonical} 
Let $\catC(X)$ be a set of measurable functions $X \to [0,1]$ on each convex space $X$ which together satisfy conditions \ref{enum:qc}, \ref{enum:comp}, and either \ref{enum:aff} or \ref{enum:exp}, below. 
\begin{enumerate} 
\item \label{enum:qc}
Each $\catC(X) \subseteq \QCon(X)$; 
\item \label{enum:comp}
$C \otimes D \in \catC(X \otimes Y)$ whenever $C \in \catC(X)$, $D \in \catC(Y)$; 
\item \label{enum:aff}
$\catC([0,1])$ contains all affine functions $[0,1] \to [0,1]$
\item \label{enum:exp}
$\catC([0,1])$ contains all exponential functions $x \mapsto \lambda^{-x}$ for $\lambda \geq 1$. 
\end{enumerate} 
Then $\catC(X) \subseteq \FCon(X)$ for all $X$. Conversely, $\FCon(X)$ satisfies all of the conditions \ref{enum:qc} - \ref{enum:exp}. 
\end{theorem}

\begin{proof}
We have seen that log-concave functions satisfy these properties. Conversely, suppose we have such a set of functions $\catC(X)$ on each convex space $X$. Fix a convex space $X$ and let $C \in \catC(X)$, with $x, y \in X$ and $p \in [0,1]$. We will show that 
\begin{equation} \label{eq:log-c-fn-example}
C(x \plusp y) \geq C(x)^p C(y)^{1-p}
\end{equation} 
so that $C$ is log-concave. If either $C(x)$ or $C(y)$ is zero this is trivial. Otherwise without loss of generality suppose $c := \frac{C(y)}{C(x)} < 1$. Suppose that $h \colon [0,1] \to [0,1]$ is a function satisfying 
\begin{equation} \label{eq:useful-condn}
\begin{cases} 
h(0) = \lambda \\ 
h(p) = \lambda c^p  \\ 
h(1) = \lambda c 
\end{cases}
\end{equation} 
for some $\lambda \in (0,1)$, and such that $C \otimes h$ is quasi-concave. Then by rescaling, for simplicity we may assume $\lambda = 1$. Then we have 
\begin{align*} 
C(x \plusp y)h(p) 
&\geq 
\min\{C(x)h(0),C(y)h(1)\}
\\&=
\min\{C(x),C(y)c\} = C(x)
\end{align*} 
Multiplying both sides by ${\frac{C(y)}{C(x)}}^{1-p}$ then yields precisely \eqref{eq:log-c-fn-example}.

Now suppose that $\catC([0,1])$ contains the exponential functions as above. Then $h(x)=c^x$ for $x \in [0,1]$ satisfies \eqref{eq:useful-condn} and by assumption $C \otimes h \in \catC(X \otimes [0,1])$, making it quasi-concave, and so we are done. 

Next suppose instead that $\catC([0,1])$ contains the affine functions. To establish the result, we will show for any $p,c \in (0,1)$ that there exists a quadratic polynomial $h$ with real roots satisfying \eqref{eq:useful-condn}, with $h([0,1]) \subseteq [0,1]$. Then by rescaling further if necessary this means that we can write $h(x) = q(x)r(x)$ where $q, r \colon [0,1] \to [0,1]$ are affine functions, i.e. of the form $x \mapsto Ax + B$. By assumption the function $C \otimes q \otimes r \colon (x,y,z) \mapsto C(x)q(y)r(z)$ belongs to $\catC(X \otimes [0,1] \otimes [0,1])$, making it quasi-concave. Since the map $y \mapsto (y,y)$ is affine, this means that $C \otimes h \colon X \otimes [0,1] \to [0,1]$ is quasi-concave also, and we are done. 
 
It remains to specify such a quadratic polynomial $h$. 
Let $ d:= \frac{c^p -c - (1-c)(1-p)}{p(1-p)}$ and define 
\[
g(x) = (1-c + dx)(1-x) + c 
=-dx^2 + (c +d - 1)x + 1
\]Then $g(0) = 1$, $g(1) =c$ and $g(p) = c^p$. Note that $d$ is positive; after rearranging the numerator this can be seen to follow from the fact that 
$p(c-1) \leq c-1 \leq c^p -1$.   

Hence $g$ has has two real roots, and since $g(0), g(1) > 0 $ these lie outside of $[0,1]$. So we can write  $g(x) = k (x-a)(b-x)$ where $a<0, b>1$ and $k > 0$. Hence $g(x) = K q(x)q'(x)$ where $q$ and $q'$ are affine functions $[0,1] \to [0,1]$, for appropriate positive scalar $K$. Thus $h(x) = q(x)q'(x)$ is our desired function in $\catC([0,1])$, and we are done. 
\end{proof}

Hence if we accept quasi-concavity (Criterion \ref{crit:QC}) and wish fuzzy concepts to be both closed under tensors and to include a few basic examples, then log-concave functions are the broadest definition we can take.

\begin{examples} \label{examples:fuzzy-concepts} 
Let us meet some examples of fuzzy concepts. 
\begin{enumerate} 
\item 
Crisp concepts $M \subseteq X$ correspond precisely to fuzzy concepts $1_M$ on $X$ taking values only in $\{0,1\}$, via  
\[
1_M(x) = \begin{cases} 1 & x \in M \\ 0 & x \notin M \end{cases} 
\]
Indeed one may see that $1_M$ is log-concave iff $M$ is convex, and measurable iff $M$ is. We will often identify a crisp concept $M \subseteq X$ with its fuzzy concept $1_M \colon X \to [0,1]$. 

\item 
Any affine measurable map $X \to [0,1]$ is concave and hence a fuzzy concept.

\item \label{enum:Gaussian-fuzzification}
Let $X$ be a normed space with metric $d$, and $P \subseteq X$ a closed crisp concept. Then for any  $\sigma^2 \geq 0$ we can define a Gaussian `fuzzification' of $P$ as the fuzzy concept 
\begin{equation} \label{eq:Gaussian-fuzzy} 
N^P_\sigma(x):= e^{-\frac{1}{2\sigma^2}d_H(x,P)^2} 
\end{equation} 
where $d_H$ denotes the Hausdorff distance $d_H(x,A) := \inf_{a\in A} d(x,a)$ for $A \subseteq X$. Here $P$ provides the `prototypical' region in which the concept takes values $1$. The concept tends to $0$ as we move away from $P$ at a rate determined by the variance $\sigma^2$. The limit case $\sigma=0$ corresponds to the crisp concept $1_P$. An example plot of such a fuzzy concept is shown in Figure 1. 

Many statistical functions besides Gaussians are log-concave also, providing alternative `fuzzification' procedures to \eqref{eq:Gaussian-fuzzy}. 

\begin{figure} \label{figure:fuzzy} 
\begin{center}
\includegraphics[scale=0.5]{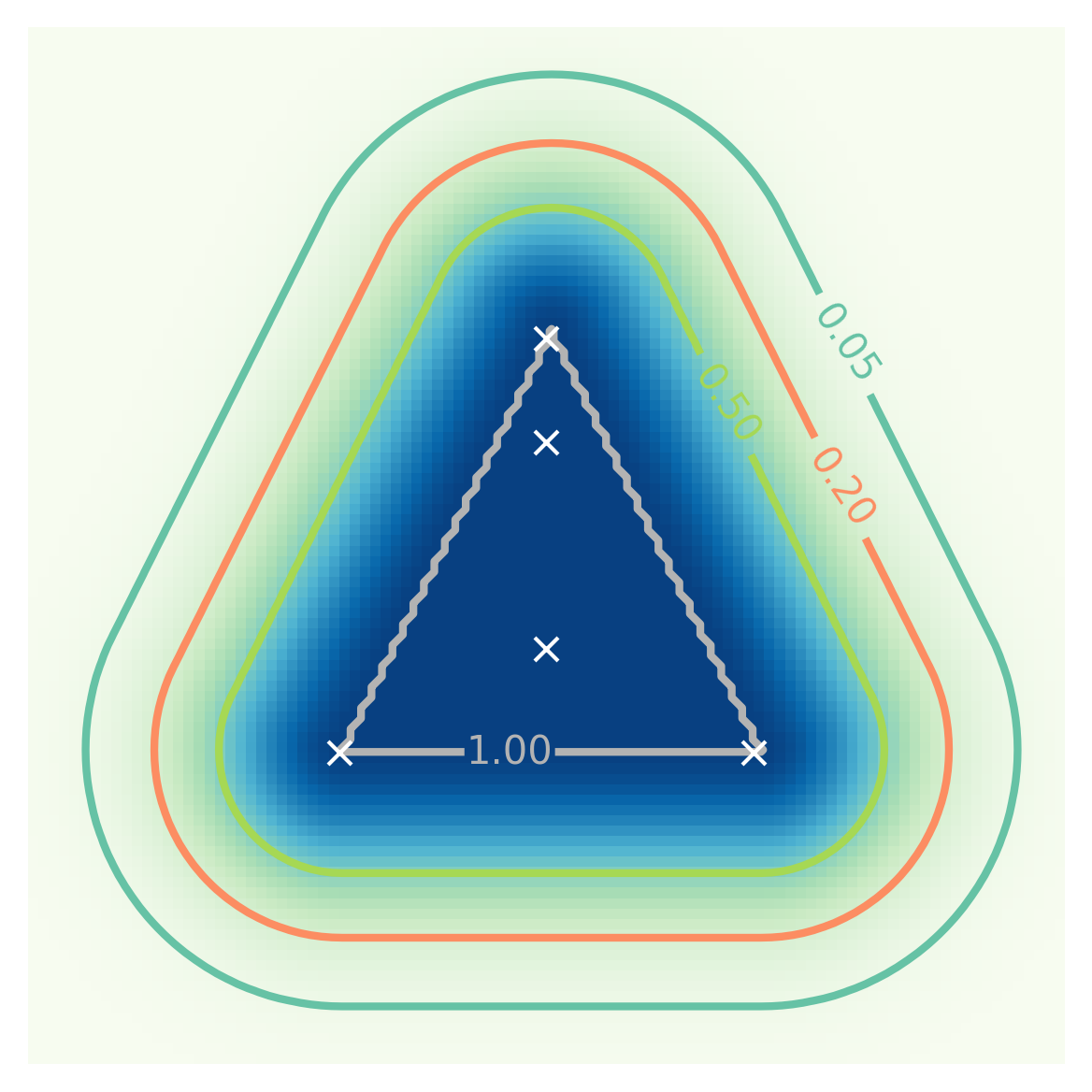}
\end{center}
\begin{caption} \ 
Visualization of a fuzzy concept in $\mathbb{R}^2$. From a set of exemplars (white crosses) we form their convex closure, yielding the crisp concept given by the inner triangle. We then form a Gaussian fuzzification as in \eqref{eq:Gaussian-fuzzy}. Each $t$-cut of the concept (within the contour lines) is a (convex) crisp concept. 
\end{caption}
\end{figure}

\item 
For any Hilbert space $\hilbH$, any quantum effect $a \in B(\hilbH)$ with $0 \leq a \leq 1$ provides an (affine) fuzzy concept $\Tr(a-)$ on the convex space of density matrices of $\hilbH$. 

\item 
Let $L$ be a finite semi-lattice viewed as a convex space with $\Sigma_L = \pset(L)$. A fuzzy concept on $L$ is a monotone map $L \to [0,1]$. 

\end{enumerate} 
\end{examples}

\section{Probabilistic Channels} \label{sec:prob-cat}

Our next goal is to introduce a category of fuzzy \emph{processes} between spaces. To do so, in this section we must briefly recall the categorical treatment of fuzzy (probabilistic) mappings, also known as `channels'. 

Recall that a finite \emph{measure} on a measurable space $(X, \Sigma_X)$ is a function $\omega \colon \Sigma_X \to \mathbb{R}$ which is additive on countable disjoint unions, with $\omega(\emptyset) = 0$. A \emph{subprobability measure} has $\omega(X) \leq 1$ while a \emph{probability measure} has $\omega(X) = 1$. These provide a general notion of `distribution' over such a space $X$. A standard approach to probability is to work in the following category, known as that of \emph{probabilistic relations} \cite{panangaden1998probabilistic}, or more abstractly as the Kleisli category of the (sub-)\emph{Giry Monad} \cite{giry1982categorical,lawvere1962category}. 
While the objects of the category are usually general measurable spaces, here we restrict to convex spaces from the outset.

\begin{definition}
In the symmetric monoidal category $\ProbC$ the objects are convex spaces and the morphisms $f \colon X \kto Y$ are \emph{channels}, also known as \emph{Markov kernels},  i.e., functions $f \colon X \times \Sigma_Y \to [0,1]$ such that 
\begin{enumerate} 
\item 
$f(x,-) \colon \Sigma_Y \to [0,1]$ is a subprobability measure on $Y$, for each $x \in X$; 
\item 
$f(-,M)\colon X \to [0,1]$ is measurable, for each $M \in \Sigma_Y$.
\end{enumerate} 
As above we write $f$ for both the morphism and function on $X \times \Sigma_Y$, and at times write $f(x) := f(x,-)$. Thus a channel sends each $x \in X$ to a `(sub)probability distribution' $f(x)$ over $Y$, in a measurable way.  Given another channel $g \colon Y \kto Z$, the composite channel $(g \circ f) \colon X \kto Z$ is defined by \[
(g \circ f)(x,M) := \int_{y \in Y} g(y,M) df(x)(y)
\] 
for each $x \in X, M \in \Sigma_Z$. The identity channel $X \kto X$ sends each $x \in X$ to the \emph{point measure} 
\[
\delta_x(M) = \begin{cases} 1 & x \in M \\ 0 & x \notin M \end{cases}
\]
The unit object $I$ is the singleton set, with $X \otimes Y = X \times Y$, the product of convex spaces. For channels $f \colon X \kto W$ and $g \colon Y \kto Z$ we define $f \otimes g \colon X \otimes Y \to W \otimes Z$ by 
\begin{equation} \label{eq:prod-meas-compn}
(f \otimes g)((x,y),(A,B)) = f(x,A) g(y,B) 
\end{equation} 
for each $x \in X, y \in Y$, $A\in \Sigma_W, B \in \Sigma_Z$. Since the measures $f(x), g(y)$ are finite, this in fact specifies $(f \otimes g)(x,y)$ over $\Sigma_{W \times Z}$ uniquely as the \emph{product measure} $f(x) \otimes g(y)$ of the measures $f(x)$ and $g(y)$.

As special cases, \emph{states} $\omega \colon I \kto X$ of $X$ may be identified with a sub-probability measures over $X$, \emph{effects} $C \colon X \kto I$ with measurable functions $C \colon X \to [0,1]$, and \emph{scalars} $I \kto I$ with probabilities $p \in [0,1]$. 

\end{definition}

\section{The Category of Log-Concave Channels} \label{sec:log-concave-channels} 

We can now generalise our notion of fuzzy concept to define a symmetric monoidal category of `fuzzy conceptual processes'. These aim to model cognitive transformations of (fuzzy) concepts. Examples include reasoning processes, metaphorical mappings between domains \cite{gardenfors2001reasoning}, and word meanings as described in the DisoCat formalism for NLP in compact \cite{coecke2010mathematical,bolt2019interacting} and more generally monoidal \cite{coecke2019mathematics,delpeuch2014autonomization} categories. 

\begin{definition}[Log-Concave Channels]
We call a channel $f \colon X \kto Y$ between convex spaces $X, Y$ \emph{log-concave} when its  kernel $f \colon X \times \Sigma_Y \to [0,1]$ is log-concave on convex subsets. That is, we have 
\begin{equation} \label{eq:lc-channel}
f(x \plusp y, A \plusp B) \geq f(x,A)^pf(y,B)^{1-p}
\end{equation}
for all $x, y \in X$ and convex $A, B \in \Sigma_Y$ for which $A \plusp B$ is measurable. 
\end{definition} 

We also call such a channel a \emph{conceptual channel}. Note that a conceptual channel $C \colon X \to I$ is precisely a fuzzy concept on $X$. Many more examples are given in the next section.

\begin{definition}[The Category $\LogCon$]
We define $\LogCon$ to be the symmetric monoidal subcategory of $\ProbC$ whose objects are convex spaces and whose morphisms are log-concave channels. 
\end{definition}

To establish that $\LogCon$ is indeed a well-defined category is non-trivial, requiring an extension of the following central result in the study of log-concave functions.

\begin{lemma}[\PL{} inequality \cite{prekopa1971logarithmic}] \label{lem:PL}  Let $0 < p < 1$ and $f,g ,h$ be non-negative measurable functions on $\mathbb{R}^n$ satisfying
\begin{equation} \label{eq:PL-Condn}
f(x \plusp y) \geq  
g(x)^ph(y)^{1-p}
\end{equation} 
for all $x,y$. 
Then 
\[
\int_{\mathbb{R}^n} f(z) dz \geq \left( \int_{\mathbb{R}^n} g(x) dx \right)^p \left( \int_{\mathbb{R}^n} h(y) dy \right)^{1-p}
\]
\end{lemma} 

We now extend this result as follows. We say that a measure $\mu$ on a measurable space $X$ is \emph{$\sigma$-finite} if $X$ can be written as a countable union of sets $X_i$ with $\mu(X_i) \leq \infty$.

\begin{theorem}[Extended \PL{} inequality] \label{thm:PL-extended}
Let $X$ be a convex space, $p \in (0,1)$, and $\mu, \nu , \omega$ be $\sigma$-finite measures on $X$ satisfying 
\begin{equation} \label{eq:three-measures-rule}
\mu(C) \geq \nu(A)^p \omega(B)^{1-p}
\end{equation}
whenever $A, B, C \in \Sigma_X$ with $C \supseteq A \plusp B$. Let $f, g, h$ be non-negative measurable functions on $X$ satisfying \eqref{eq:PL-Condn} for all $x,y \in X$. Then
\begin{equation} \label{eq:PL-extended}
\left(\int_{X} f \ d\mu \right) 
\geq 
\left(\int_X g \ d\nu \right)^p 
\left(\int_X h \ d\omega \right)^{1-p}
\end{equation} 
Further, if $f, g, h$ are quasi-concave then we need only require \eqref{eq:three-measures-rule} for $A,B,C\in \Sigma_X$ which are convex.
\end{theorem}

\begin{proof}
First observe that if $t, t' \in \mathbb{R}$ with $g(x) > e^t$ and $h(y) > e^{t'}$ then   
\[
f(x \plusp y) \geq g(x)^p h(y)^{1-p} > (e^t)^p (e^{t'})^{(1-p)} = e^{t \plusp t'}
\]
Hence we have $f^{-1}(e^{t \plusp t'}, \infty) \supseteq g^{-1}(e^{t}, \infty) \plusp h^{-1}(e^{t'}, \infty)$. Now by assumption on the measures 
\begin{align*}
\mu(f^{-1}(e^{t \plusp t'}, \infty)) e^{t \plusp t'} 
&\geq 
\nu(g^{-1}(e^{t}, \infty))^p \omega(h^{-1}(e^{t'}, \infty))^{1-p} e^{t \plusp t'} 
\\&=
(\nu(g^{-1}(e^{t}, \infty)) \cdot e^{t})^p (\omega(h^{-1}(e^{t'}, \infty)) \cdot  e^{t'})^{1-p} 
\end{align*}  
for all $t, t' \in \mathbb{R}$. Applying a well-known consequence of Fubini's theorem, then integral substitution with $u = e^v$, and then finally the one-dimensional \PL{} inequality (Lemma \ref{lem:PL}) we have 
\begin{align*} 
 \int_X f \ d\mu 
&=
\int^{\infty}_0 \mu(f^{-1}(u,\infty)) du
=
\int^{\infty}_{-\infty} \mu(f^{-1}(e^{v}, \infty)) \cdot e^{v} dv  
\\ &\geq 
\left( \int^{\infty}_{-\infty} \nu(g^{-1}(e^{t}, \infty)) \cdot e^{t} dt  \right)^p 
\left( \int^{\infty}_{-\infty} \omega(h^{-1}(e^{t'}, \infty)) \cdot e^{t'} dt'  \right)^{1-p} 
\\ &= 
\left( \int_X g \ d\nu \right)^p 
\left( \int_X h \ d\omega \right)^{1-p}
\end{align*} 
as required. For the final statement observe that if $f, g, h$ are quasi-concave then each subset $r^{-1}(e^t, \infty)$ for $r=f,g,h$ will be convex.
\end{proof} 

\begin{remark} 
The above provides an alternative proof of the \PL{} inequality from only its one-dimensional form and log-concavity of the Lebesgue measure (though \PL{} is typically used to establish the latter fact in the first place). It would be interesting to explore whether Theorem \ref{thm:PL-extended} provides any novel applications of the inequality to more general convex spaces.
 \end{remark}

We now reach our main result. 

\begin{theorem} \label{thm:LogCon-anSMC}
$\LogCon$ is a well-defined symmetric monoidal subcategory of $\ProbC$.
\end{theorem}

\begin{proof}
All identities and coherence isomorphisms are channels $X \to Y$ of the form $x \mapsto \delta_{g(x)}$ where $g$ is an affine measurable map, making them log-concave. Indeed, for any given $g$, this channel will be log-concave iff for all $A, B \in \Sigma_Y$, we have $g(x) \in A, g(y) \in B \implies g(x \plusp y) \in A \plusp B$. Taking $A=\{g(x)\}, B=\{g(y)\}$ shows this is equivalent to $g$ being affine. 

It now suffices to show for any log-concave $f \colon X \kto Y$, $g \colon Y \kto Z$ and any convex space $W$ that the channel $h := (g \circ f) \otimes \id{W} \colon X \otimes W \to Z \otimes W$ is log-concave. From the definition of $\ProbC$ one may see that for all $(x,w) \in X \times W$ and measurable $E \subseteq Z \times W$ we have 
\[
h((x,w),E) = \int_{y \in Y} g(y,E^w) df(x,y)
\]
where $E^w := \{z \in Z \mid (z,w) \in E\}$. Let $x,x' \in X$, $w,w' \in W$ and $p \in (0,1)$. By definition the measures  $f(x \plusp x'), f(f(x')$ are all finite and satisfy 
\[
f(x \plusp x',C) \geq f(x,A)^pf(x',B)^{1-p}
\] 
whenever $A, B, C \in \Sigma_{Y}$ are convex with $C \supseteq A \plusp B$. Now let $D, E \subseteq Z \times W$ be convex and measurable and suppose that $F := D \plusp E$ is measurable also. Then $D^w$, $E^{w'}$ will be convex also. Note that if $(x,w) \in D$ and $(x',w') \in E$ then $(x \plusp x', w \plusp w') \in F$. Hence $F^{w \plusp w'} \supseteq C^w \plusp D^{w'}$.  Since $g$ is log-concave, we conclude that for each $y, y' \in Y$ we have 
\begin{align*}
g(y\plusp y',C^{w \plusp w'}) &\geq g(y \plusp y',A^w \plusp B^{w'})\\
 &\geq 
g(y,A^w)^p g(y',B^{w'})^{1-p} 
\end{align*}
Noting also that $g(-,C^{w \plusp w'})$, $g(-,A^w)$ and $g(-,B^w)$ are all log-concave and hence quasi-concave functions, we may apply the final statement of Theorem \ref{thm:PL-extended} with
$\mu = f(x \plusp x'), \nu = f(x), \omega = f(x')$ to give 
\begin{align*}
h((x \plusp x', w \plusp w'),F)
&=
\int_{y \in Y}
g(y,F^{w \plusp w'})
df(x \plusp x',y)
\\ 
&\geq 
\left( 
\int_{y \in Y} 
g(y,D^w) df(x,y)
\right)^p 
\left( 
\int_{y \in Y} 
g(y,E^{w'}) df(x',y)
\right)^{1-p}
\\ 
&= 
h((x,w), D)^p 
h((x',w'),E)^{1-p}
\end{align*}
Hence $h$ is a log-concave channel as required. 
\end{proof}

\begin{remark} 
We could instead have defined log-concave channels $f$ to satisfy \eqref{eq:lc-channel} for arbitrary (not necessarily convex) $A, B \in \Sigma_Y $ with $A \plusp B$ measurable. Such `fully log-concave' channels form a monoidal subcategory of $\LogCon$, with the same proof as Theorem \ref{thm:LogCon-anSMC}. 
\end{remark}

We can also extend Theorem \ref{thm:fuzzy-concepts-canonical} to show that our definition of conceptual channel is not arbitrary, but that $\LogCon$ is `the largest' subcategory of $\ProbC$ whose effects can form fuzzy concepts. Let us call a convex space \emph{well-behaved} if for all convex measurable subsets $A, B$, the convex set 
\begin{equation} \label{eq:conv-set-exmpl}
\{(a \plusp b, p) \mid a \in A, b\in B, p\in [0,1]\} \subseteq X \times [0,1]
\end{equation} 
is measurable. We conjecture that every normed space, with its Borel $\sigma$-algebra, is well-behaved. 

\begin{remark} 
Note that well-behavedness would fail even in $\mathbb{R}$ if we did not require $A, B$ to be convex; there are (Borel) measurable sets $A, B \subseteq \mathbb{R}$ for which $A + B$ and hence \eqref{eq:conv-set-exmpl} are not measurable \cite{erdHos1970sum}.
\end{remark}

\begin{theorem}[Log-Concave Channels are Canonical] \label{thm:LogCon-Channels-Canonical} 
Let $\catC$ be a symmetric monoidal subcategory of $\ProbC$ containing only well-behaved convex spaces, as well as the space $[0,1]$, and for each object $X$ write $\catC(X) := \catC(X,I)$. Suppose that either of the following hold. 
\begin{enumerate} 
\item \label{enum:effects-fuzzy}
$\CC(X) \subseteq \catC(X) \subseteq \FCon(X)$ for all $X$; 
\item \label{enum:effects-interesting}
$\CC(X) \subseteq \catC(X) \subseteq \QCon(X)$ for all $X$ and  $\catC([0,1])$ either contains all affine functions or contains all exponential functions.
\end{enumerate} 
Then there are symmetric monoidal inclusions 
$\catC \hookrightarrow \LogCon \hookrightarrow \ProbC$.
\end{theorem} 

\begin{proof}
Since $\catC$ is a monoidal subcategory of $\ProbC$, all effects in $\catC$ are measurable and condition \eqref{enum:comp} of Theorem \ref{thm:fuzzy-concepts-canonical} holds. Hence by Theorem \ref{thm:fuzzy-concepts-canonical} we have \eqref{enum:effects-interesting} $\implies$ \eqref{enum:effects-fuzzy}. We now show that \eqref{enum:effects-fuzzy} ensures that every $f \colon X \kto Y$ in $\catC$ is log-concave. Let $A, B \in \Sigma_Y$ be convex. Then defining $C$ to be the set \eqref{eq:conv-set-exmpl}, by assumption $1_C$ is an effect on $Y \times [0,1]$ in $\catC$. Hence the effect $D = 1_C \circ (f \otimes \id{})$ on $X \times [0,1]$ belongs to $\catC$ also, and must be log-concave. Thus for any $x,y\in X$ 
\[
f(x \plusp y,A \plusp B) 
= D(x,\plusp y, p) 
\geq 
D(x,1)^pD(y,0)^{1-p} 
=f(x,A)^pf(y,B)^{1-p}
\]
making $f$ log-concave.
\end{proof}

\begin{remark} 
This proof shows that log-concave channels satisfy an analogue of `complete positivity'. If a channel $f$ is such that each channel $f \otimes \id{X}$ preserves fuzzy concepts under post-composition for all objects $X$ (or even just $X=[0,1]$), then $f$ must be log-concave. 
\end{remark}

\section{Examples of Log-Concave Channels} \label{sec:log-chan-examples}

To make sense of the definition of log-concave channel and illustrate working in the category $\LogCon$ we now give numerous examples of its morphisms. We use the \emph{graphical calculus} for symmetric monoidal categories, in which morphisms $A \to B$ are boxes with lower input wire $A$ and upper output wire $B$ (read bottom to top), with $\id{I}$ corresponding to the empty diagram \cite{selinger2010survey}.  

\begin{enumerate}[leftmargin=0.5cm]
\item \textbf{Effects.} Scalars $I \kto I$ are values $p \in [0,1]$, and fuzzy concepts on $X$ correspond precisely to effects 
\[
\scalebox{1.0}{\input{./figures/effect.tikz}} 
\] 

\item \textbf{States.} A state 
\[
\scalebox{1.0}{\input{./figures/state.tikz}} 
\]
is a sub-probability measure $\omega$ on $X$ satisfying 
\begin{equation} \label{eq:subprob-lc} 
\omega(A \plusp B) \geq \omega(A)^p\omega(B)^{1-p}
\end{equation} 
for all convex $A, B \in \Sigma_X$ for which $A \plusp B \in \Sigma_X$. Measures for which this holds for \emph{arbitrary} $A, B \in \Sigma_X$ are called \emph{log-concave measures}, and are well-studied with log-concave functions \cite{klartag2005geometry}. Thus states on $X$ are essentially log-concave sub-probability measures. 

Given a fuzzy concept $C$ on $X$, the scalar 
\[
\scalebox{1.0}{\input{./figures/state-effect.tikz}} = \int_{x \in X} C(x) d\omega(x) \in [0,1]
\]
is the extent to which the concept $C$ is deemed to hold over the `distribution of inputs' $\omega$.

\item \textbf{States from densities.}
When $X$ is (a convex measurable subset of) $\mathbb{R}^n$, it is well-known that log-concave measures $\omega$ correspond precisely to log-concave \emph{densities}, as follows. If $\rho \colon X \to \Rplus$ is a measurable log-concave function we may define a log-concave measure $\omega := \rho\lambda^X$ on $X$ by 
\begin{equation} \label{eq:density}
\omega(A) = \int_A \rho d\lambda^X
\end{equation} 
for each $A \in \Sigma_X$, where $\lambda^X$ is the Lebesgue measure on $X$. Conversely, every log-concave measure on $X$ is of the form  $\omega = i \circ \omega'$ where $Y$ is a measurable convex subset, $\omega'=\rho\lambda^Y$ for a log-concave density $\rho$ on $Y$, and $i$ is the inclusion $Y \hookrightarrow X$. Specifically $Y$ is the affine closure of $\omega$'s \emph{support}. 

Many standard probability distributions on $\mathbb{R}^n$ form states in $\LogCon$, including:

\begin{enumerate} 
\item Each point measure $\delta_x$;
\item Each uniform distribution over a compact convex compact subset $C$, with density $1_C(x)$ and measure $\omega(A) = \frac{\lambda(A \cap C)}{\lambda(C)}$ where $\lambda$ is the Lebesgue measure; 
\item 
Each multivariate Gaussian distribution, which (on its affine support) has log-concave density 
\begin{equation} \label{eq:Gaussian} 
\rho(x) = \frac{1}{\kappa} e^{-\frac{1}{2}(x - \mu)^{\mathsf{T}}\Sigma^{-1}(x -\mu)}
\end{equation}  
with mean $\mu \in X=\mathbb{R}^n$, covariance matrix $\Sigma$ and normalisation $\kappa = \sqrt{((2\pi)^n\det(\Sigma)}$; 
\item 
The logistic, extreme value, Laplace and chi distributions on $\mathbb{R}$, all with log-concave densities. 
\end{enumerate}

\item \textbf{Markov category maps.} Each convex space $X$ comes with log-concave \emph{copying} and \emph{discarding} channels which form a commutative comonoid
 \[
\scalebox{1.0}{\input{./figures/copy-delete.tikz}}
 \]
and are defined by $x \mapsto \delta_{(x,x)}$ and $x \mapsto 1$ for all $x$, respectively. 

 These makes $\LogCon$ a \emph{copy/discard category} \cite{cho2019disintegration}, and hence its subcategory of discard-preserving maps a \emph{Markov category} \cite{fritz2020representable}. The presence of discarding tells us that \emph{marginals} of log-concave channels are again log-concave, which is well-known for log-concave measures.

\item \textbf{Conceptual updates.} \label{enum:concept-updates} 
Copying lets us turn any fuzzy concept $C$ into an `update by $C$' map (left-hand below), as well as point-wise multiply any pair of fuzzy concepts $C, D$.
\begin{equation} \label{eq:concept-update}  
\scalebox{1.0}{\input{./figures/update.tikz}} :: x \mapsto C(x) \delta_x \quad \qquad \qquad \scalebox{1.0}{\input{./figures/CD.tikz}} :: x \mapsto C(x)D(x)
\end{equation}

\item \textbf{Affine maps.} Any \emph{partial affine map} $f \colon X \kto Y$, meaning a convex measurable subset $\dom(f) \subseteq X$ and measurable affine map $f \colon \dom(f)\to Y$, induces a log-concave channel $\hat f \colon X \kto Y$ with
\[
\hat f(x) = \begin{cases} \delta_{f(x)} & x \in \dom(f) \\ 0 & \text{otherwise} \end{cases} 
\]
We can characterise these channels as follows.
\begin{lemma} \label{lem:pConv}
A channel $f \colon X \to Y$ is of the form $\hat g$ for a partial affine map $g \colon X \to Y$ iff it is log-concave and \emph{crisp} in the sense that each $f(x)$ is either zero or a point measure. 
\end{lemma} 
\begin{proof} 
Any crisp $f \colon X \kto Y$ has $f = \hat g$ for the partial function $g \colon X \relto Y$ with $f(x)=\delta_{g(x)}$ whenever $f(x)$ is defined. Then $\dom(g) = f(-,Y)^{-1}(1)$ is measurable, and for each $M \in \Sigma_Y$, so is the set $f(-,M)^{-1}(1) = \dom(g) \cap g^{-1}(M)$. Hence $g$ is a measurable function on a measurable domain. Log-concavity makes $\dom(g)$ convex and is then equivalent to $g$ being affine, just as at the start of the proof of Theorem \ref{thm:LogCon-anSMC}.
\end{proof} 

Note that all these crisp maps are \emph{deterministic}, in the sense of Markov categories \cite{fritz2020synthetic}, meaning that they satisfy
\[
\scalebox{1.0}{\input{./figures/crisp.tikz}}
\]

\item \textbf{Convolutions.} For any pair of log-concave channels $f, g \colon X \kto Y$ between vector spaces we may define their \emph{convolution} $f \star g$ as the log-concave channel 
\begin{equation} \label{eq:convolution} 
\scalebox{1.0}{\input{./figures/convolution.tikz}} 
\end{equation}
where $+$ is the monoid $(x,y) \mapsto x + y$. When interpreting $f$ and $g$ as sending each $x \in X$ to random variables over $Y$, $f \star g$ sends each element to the sum of these random variables. 

\item \textbf{Noisy maps.}
As a case of the previous example, given any (measurable) partial affine map $f \colon X \to Y$, now viewed as a channel, and any state $\nu$ of $Y$, we can form a log-concave channel 
\begin{equation} \label{eq:noise}  
\scalebox{1.0}{\input{./figures/add-noise.tikz}}
\end{equation}
which sends $(x,A) \mapsto \nu(A - f(x))$. If $\nu$ models `random noise' over the space $Y$, then this channel describes a random variable $y = fx + \nu$ in terms of input $x \in X$.  Considering spaces $\mathbb{R}^n$ and maps \eqref{eq:noise} where $f$ is linear and $\nu$ is a Gaussian (noise) probability measure yields the symmetric monoidal category $\Gauss$ of Gaussian probability theory from  \cite{fritz2020synthetic}. Thus $\Gauss \hookrightarrow \LogCon$. 

\item \textbf{Channels from densities.} Let $X, Y$ be convex measurable subsets of $\mathbb{R}^n, \mathbb{R}^m$ respectively, and $\rho \colon X \times Y \to \Rplus$ be a measurable log-concave function  such that $\int_Y\rho(x,y)dy \leq 1$ for each $x \in X$, where $dy$ denotes the Lebesgue measure on $Y$. Then we may define a log-concave channel by 
\begin{equation} \label{eq:density-channel}
f(x,A) = \int_{A} \rho(x,y) dy
\end{equation} 
for each $A \in \Sigma_Y$. This follows from the usual \PL{} inequality in $\mathbb{R}^n$  (Lemma \ref{lem:PL}).  It would be interesting to find a converse to this result, analogous to that for states. 
\end{enumerate} 

\section{Toy Application: Reasoning in Food Space} \label{sec:toy-example}

In closing we demonstrate a toy example of conceptual reasoning in $\LogCon$, returning to our example of `food space' $F = C \otimes T$ from Example \ref{ex:conceptual-spaces} \eqref{enum:food-space-initial}, based on \cite{bolt2019interacting}. As in that example, let us first define a crisp concept $\text{Green} = B^{\epsilon}_g$ of radius $\epsilon=0.1$ in $C$ around pure green $g=(0,1,0)$. We can extend this to a crisp concept on the whole of $F$ via
\[
\scalebox{1.0}{\input{./figures/green-extended.tikz}}
\] 
In the same way we define crisp concepts `Yellow', `Sweet' and `Bitter' on $F$.

Now suppose an agent wishes to learn the concept of `banana' from a set of exemplars in $F$ containing a banana they conceptualise as yellow and sweet, as well as another they deem to be green and bitter. They form a crisp concept $B$ by taking the convex closure of these concepts
\begin{equation} \label{eq:banana-learn}
\scalebox{1.0}{\input{./figures/banana.tikz}}
\end{equation} 
where $C \vee D = \overline{C \cup D}$ is the convex closure, the join in the partial inclusion order on crisp concepts of $F$.

\paragraph{Fuzzifying concepts} 
Since they are uncertain about the definition of their new concept `banana', the agent may wish to replace their concept with a fuzzy one. They can convert all of their crisp concepts into fuzzy ones using the `Gaussian fuzzification' of Example \ref{examples:fuzzy-concepts} \eqref{enum:Gaussian-fuzzification}. For example, we define a fuzzification `banana' of the crisp concept `Banana' with variance $\sigma_{Ba}^2$ as 
\[
\scalebox{1.0}{\input{./figures/banana-fuzzy.tikz}} := N^{\text{Banana}}_{\sigma_{Ba}}  :: 
(c,t) \mapsto e^{-\frac{1}{2\sigma^2}d_H((c,t),\text{Banana})^2} 
\]
We define fuzzy concepts `green', `yellow', `bitter' and `sweet' via $\sigma_G$, $\sigma_{Y}, \sigma_{Bi}, \sigma_{S}, \sigma_{Ba}$ similarly.

\paragraph{Combining fuzzy concepts} We can combine any of our fuzzy concepts using the copying maps, as in \eqref{eq:concept-update}. For example, we can define a fuzzy concept `green banana' as  
\[
\scalebox{1.0}{\input{./figures/greenbanana.tikz}} :: x \mapsto \text{green}(x) \text{banana}(x)
\] 
In Figure 2 we plot some examples of composite fuzzy concepts on the food space $F$. 

\begin{figure} \label{figure:concepts} 
\begin{center}
\includegraphics[scale=0.35]{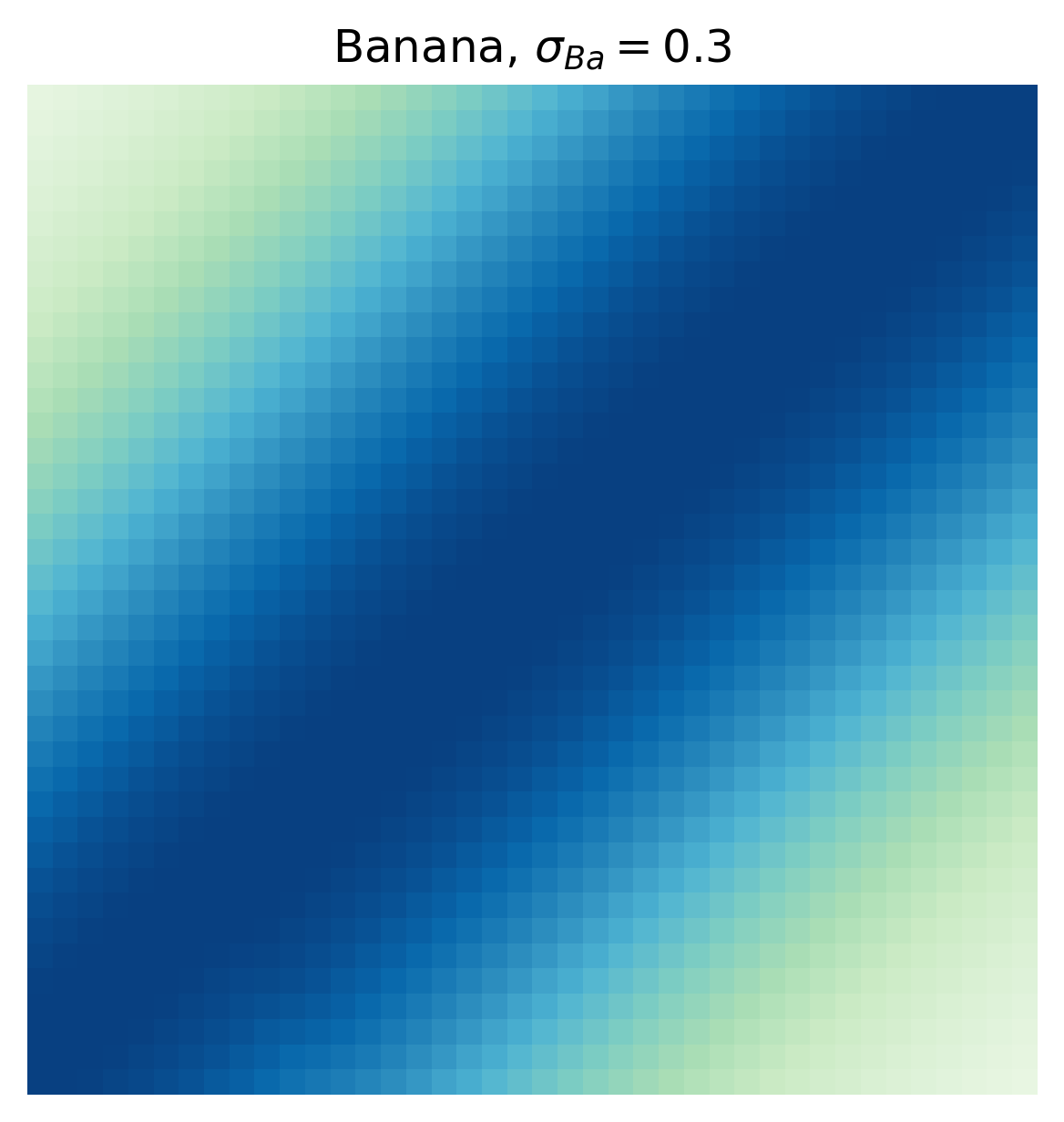}
\includegraphics[scale=0.35]{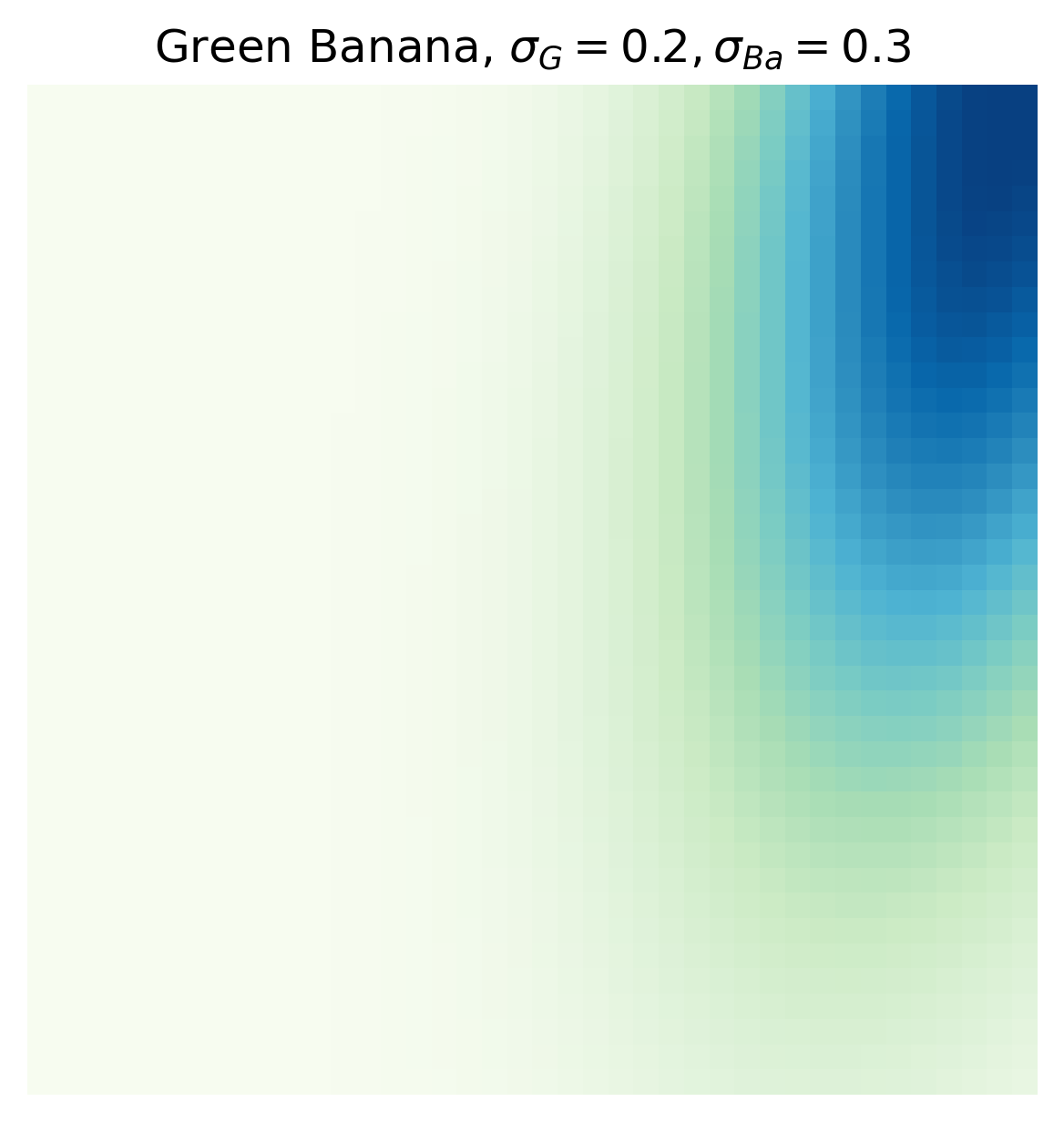}
\includegraphics[scale=0.35]{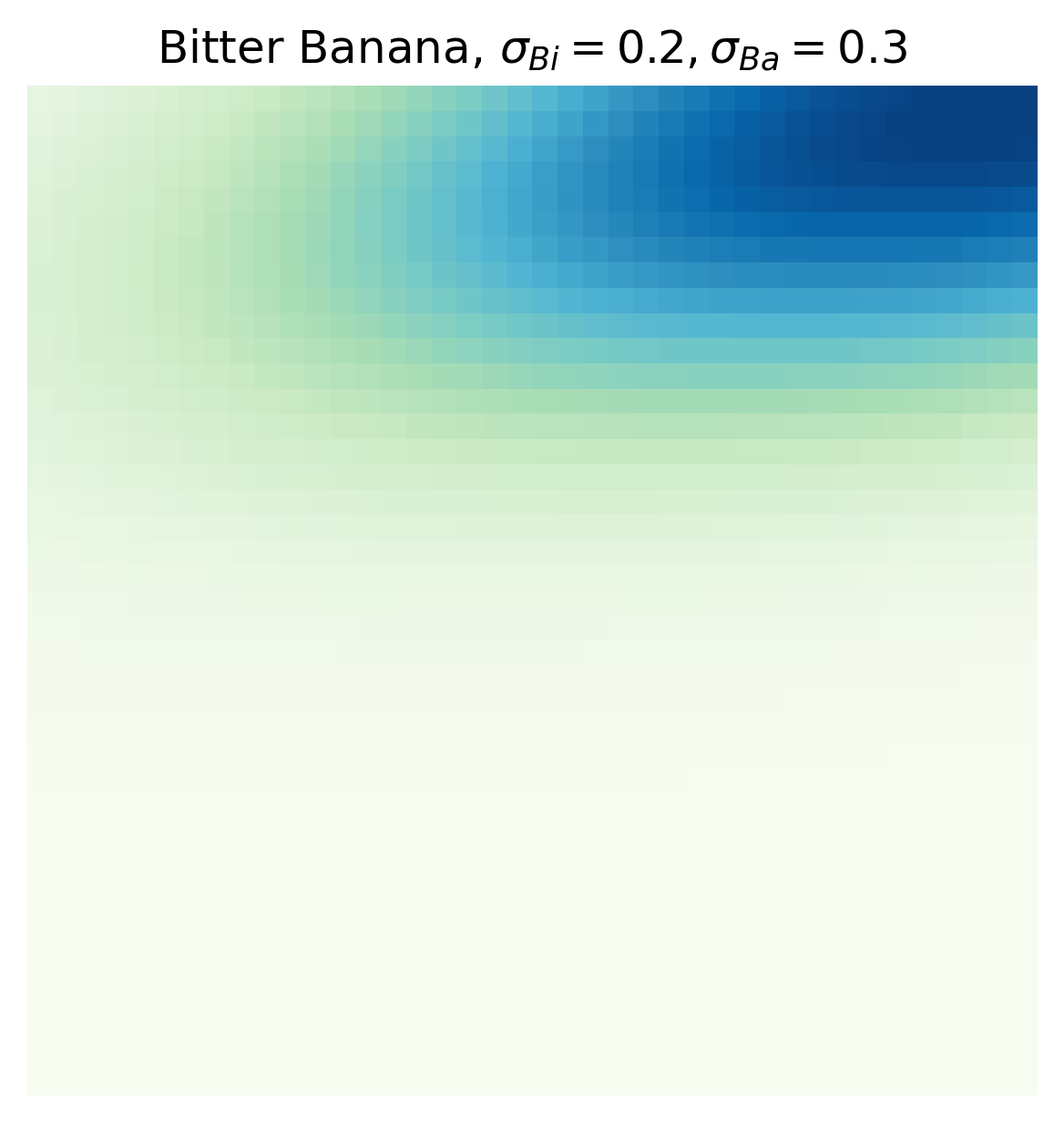}
\end{center}
\begin{center}
\includegraphics[scale=0.35]{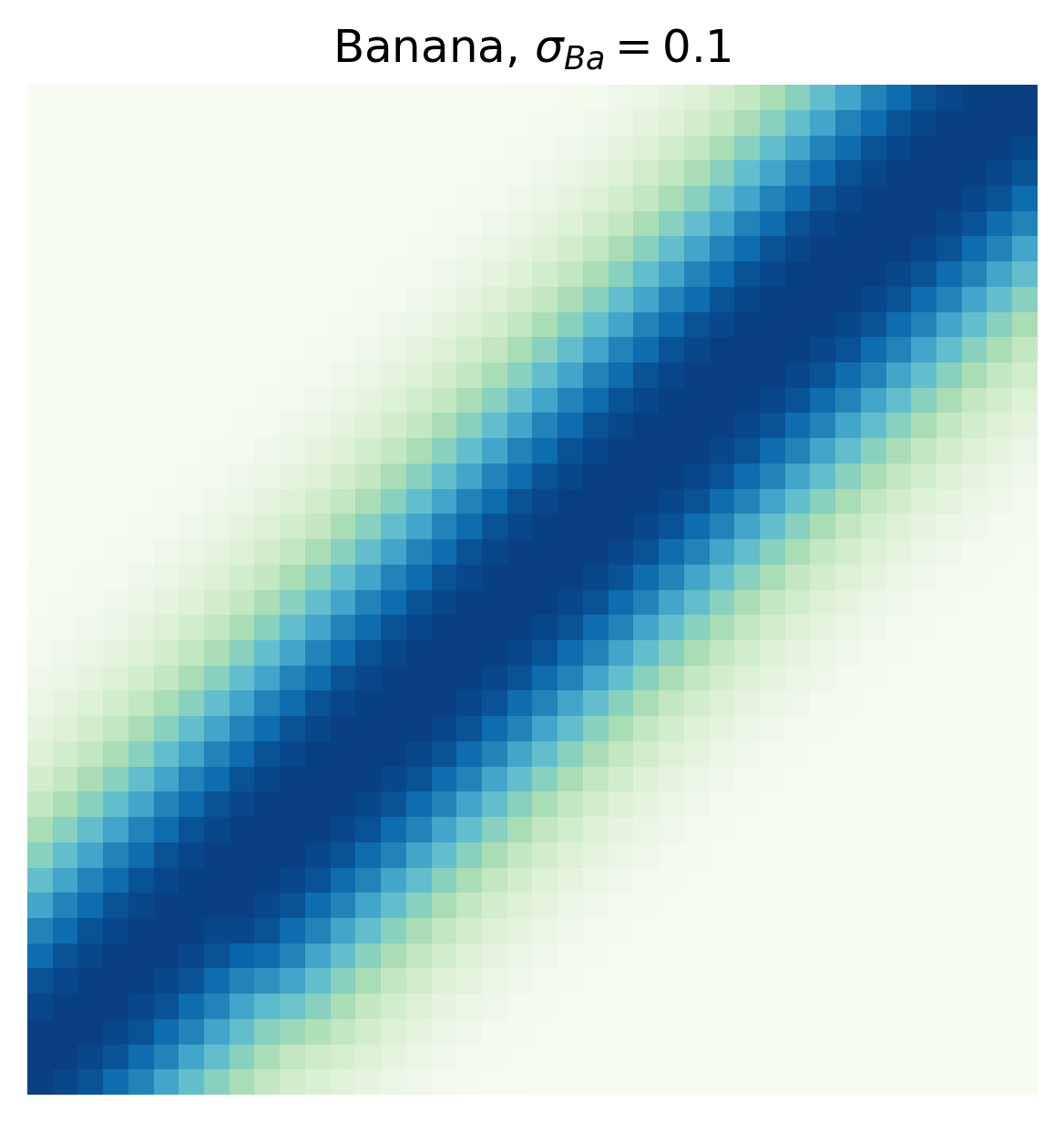}
\includegraphics[scale=0.35]{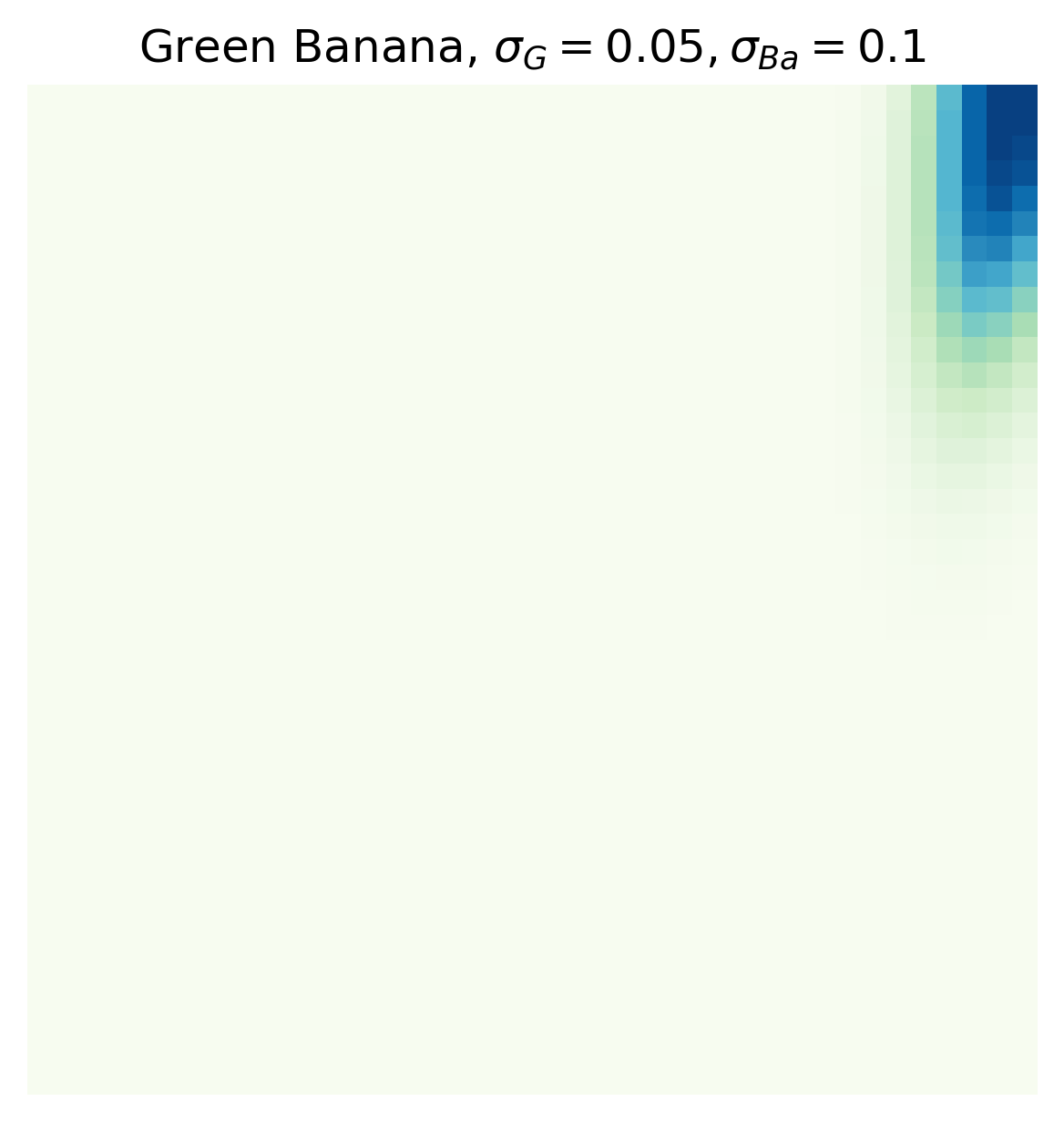}
\includegraphics[scale=0.35]{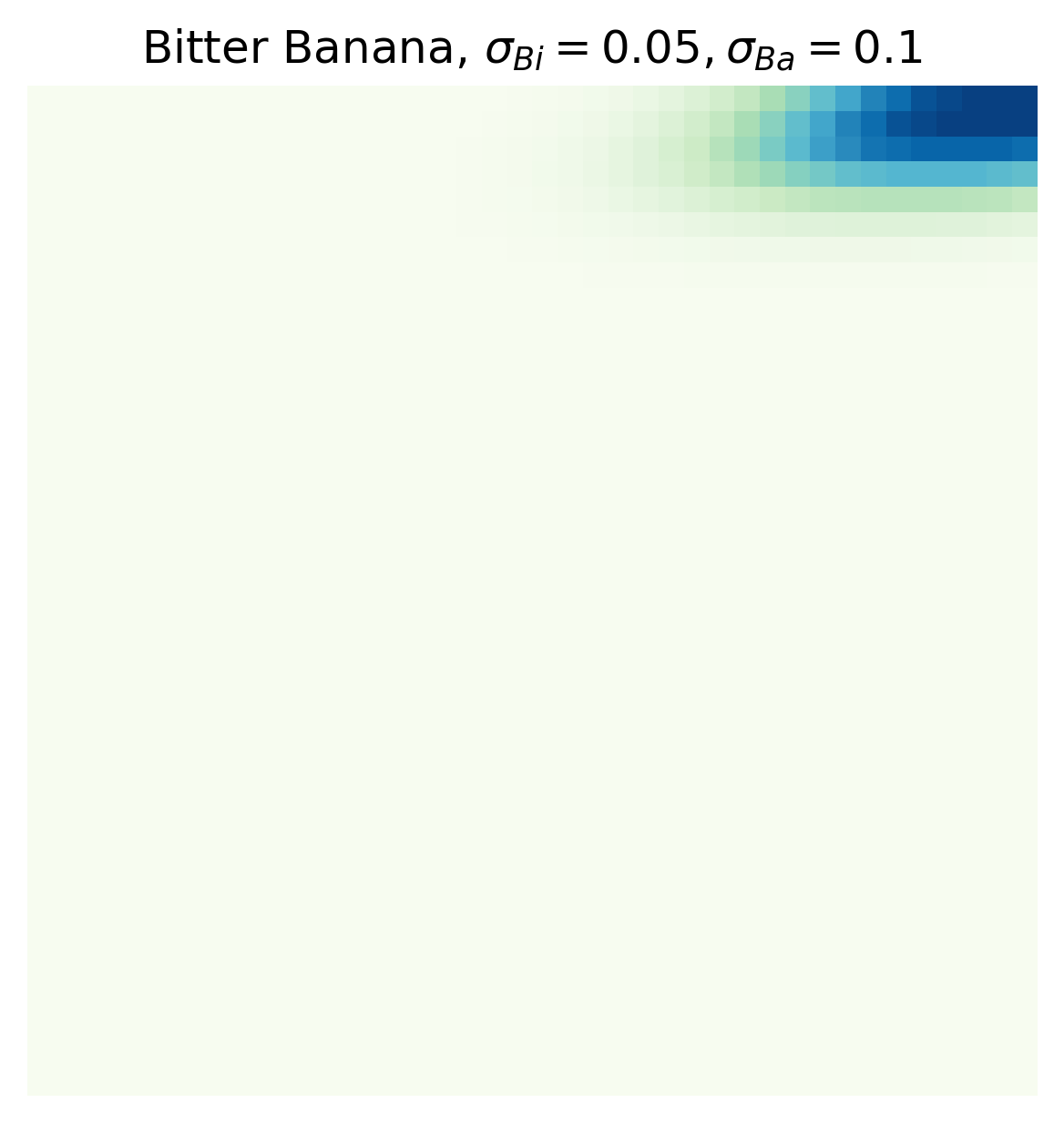}
\end{center}
\begin{caption} \ 
Fuzzy concepts in food space, for differing variance parameters, plotted over the unit square $[\text{yellow},\text{green}] \times [\text{sweet},\text{bitter}] \subseteq F$. Decreasing variance increases the crispness of the concepts. Values range from dark blue (1) through green to white (0).
\end{caption}
\end{figure}

\paragraph{A taste-colour channel} 
As a first example of conceptual reasoning beyond simply combining  concepts, we consider an example of a `metaphorical' mapping between domains. Consider the channel from tastes to colours defined by 
\[
\scalebox{1.0}{\input{./figures/metaphor.tikz}}
\] 
where $\discardflip{C}$ is the uniform (Lebesgue) measure over $C$. This channel transforms any concept on colours into one tastes via precomposition. For example we can interpret the concept of  `tasting yellow' as \[
\scalebox{1.0}{\input{./figures/tastes-yellow-reduced.tikz}} :: t \mapsto \frac{1}{\lambda(C)}\int_{c \in C} \text{yellow}(c)\text{banana}(c,t) d\lambda(c)
\] 
where $\lambda$ denotes the Lebesgue measure on $C$.

\paragraph{Future applications} 
In future it would be interesting to explore more sophisticated examples of conceptual channels, including those with a linguistic interpretation as metaphors. It would be desirable to extend the learning process \eqref{eq:banana-learn} to give a `join' on fuzzy concepts, rather than merely crisp ones. Besides point-wise multiplication, we should also aim to describe further ways of combining concepts, which for example account for the `pet fish' phenomenon \cite{fodor1996red,coecke2015compositional}.

\bibliographystyle{eptcs}
\bibliography{cog}

\end{document}

%% file: preamble.tex
\newcommand{\Gardenfors}{G\"ardenfors}

\newcommand{\dom}{\mathsf{dom}}

\usepackage{amsmath}
\usepackage{amsthm}
\usepackage{amssymb}
\usepackage{stmaryrd}  

\newcommand{\PL}{Pr\'ekopa-Leindler}
\newcommand{\ProbC}{\cat{Prob}} 

\newcommand{\Rplus}{\mathbb{R}_{\geq 0}}
\newcommand{\plusp}{+_p}

\newcommand{\kto}{\to} 

\newcommand{\Bechberger}{Bechberger{}}
\newcommand{\Kuhnberger}{Kuhnberger{}}

\theoremstyle{definition}
\newtheorem{Def}{Definition}
\newtheorem{definition}[Def]{Definition}
\newtheorem{theorem}[Def]{Theorem}

\newtheorem{lemma}[Def]{Lemma}

\newtheorem{remark}[Def]{Remark}

\newtheorem{criterion}[Def]{Criterion}

\newtheorem{examples}[Def]{Examples}

\usepackage{tikz-cd}

\newcommand{\pset}{\mathbb{P}}

\usepackage{stackengine}

\newcommand{\cat}[1]{\ensuremath{\mathbf{#1}}}
\newcommand{\catC}{\cat{C}}

\newcommand{\id}[1]{\ensuremath{\mathrm{id}_{#1}}}

\newcommand{\ConvRel}{\cat{ConvRel}}

\newcommand{\LogCon}{\cat{LCon}}

\newcommand{\hilbH}{\mathcal{H}} 

\newcommand{\discardflip}[1]{\ensuremath{\tinygroundflipnew_{#1}}}

\usepackage{tikz,xypic}
\usetikzlibrary{decorations.pathreplacing,decorations.markings,arrows.meta,backgrounds}
\pgfdeclarelayer{edgelayer}
\pgfdeclarelayer{nodelayer}
\pgfsetlayers{background,edgelayer,nodelayer,main}
\tikzstyle{whitedot}=[circle, draw=black, fill=white, inner sep=.4ex]
\tikzstyle{none}=[inner sep=0mm]

\usepackage{tikz,xypic}

\usetikzlibrary{decorations.pathreplacing,decorations.markings,arrows.meta,backgrounds}
\pgfdeclarelayer{edgelayer}
\pgfdeclarelayer{nodelayer}
\pgfsetlayers{background,edgelayer,nodelayer,main}

\tikzstyle{cdot}=[circle, draw=black, fill=black!25, inner sep=.4ex] 
\tikzstyle{bigdot}=[dot, inner sep=0pt]
\tikzstyle{whitedot}=[circle, draw=black, fill=white, inner sep=.4ex]
\tikzstyle{greydot}=[circle, draw=black, fill=black!25, inner sep=.4ex] 
\tikzstyle{blackdot}=[circle, draw=black, fill=black, inner sep=.4ex]
\tikzset{arrow/.style={decoration={
    markings,
    mark=at position #1 with \arrow{>[length=2pt, width=3pt]}},
    postaction=decorate},
    reverse arrow/.style={decoration={
    markings,
    mark=at position #1 with {{\arrow{<[length=2pt, width=3pt]}}}},
    postaction=decorate}
}

\newcommand\relto[1][]{\mathbin{\smash{
\begin{tikzpicture}[baseline={([yshift=-1pt]
current bounding box.south)}]
    \node (A) at (0,0) [inner xsep=0pt, inner ysep=1pt, minimum width=0.15cm] {\ensuremath{\scriptstyle #1}};
    \draw [->, line width=0.4pt, line cap=round]
        ([xshift=-2.5pt] A.south west)
        to ([xshift=3pt] A.south east);
    \draw [line width=.4pt] ([xshift=-1pt,yshift=-2pt]A.south) to ([xshift=1pt,yshift=2pt]A.south);
\end{tikzpicture}}}}

\newenvironment{pic}[1][] {\begin{aligned}\begin{tikzpicture}[scale=2.0, font=\tiny,#1]}{\end{tikzpicture}\end{aligned}} 

\newif\ifvflip\pgfkeys{/tikz/vflip/.is if=vflip}
\newif\ifhflip\pgfkeys{/tikz/hflip/.is if=hflip}
\newif\ifhvflip\pgfkeys{/tikz/hvflip/.is if=hvflip}

\newenvironment{picc}[1][]
{\begin{aligned}\begin{tikzpicture}[font=\tiny,#1]}
{\end{tikzpicture}\end{aligned}}

\newlength\minimummorphismwidth
\newlength\stateheight
\newlength\minimumstatewidth
\newlength\connectheight
\tikzset{colour/.initial=white}

\tikzstyle{pure}=[line width=.7pt]

\makeatletter

\pgfdeclareshape{groundd}
{
    \savedanchor\centerpoint
    {
        \pgf@x=0pt
        \pgf@y=0pt
    }
    \anchor{center}{\centerpoint}
    \anchorborder{\centerpoint}

    \anchor{north}
    {
        \pgf@x=0pt
        \pgf@y=0.16\stateheight
    }
    \anchor{south}
    {
        \pgf@x=0pt
        \pgf@y=0pt
    }
    \saveddimen\overallwidth
    {
        \pgfkeysgetvalue{/pgf/minimum width}{\minwidth}
        \pgf@x=\minimumstatewidth
        \ifdim\pgf@x<\minwidth
            \pgf@x=\minwidth
        \fi
    }
    \backgroundpath
    {
        \begin{pgfonlayer}{main} 
        \pgfsetstrokecolor{black}
        \pgfsetlinewidth{1.25pt}
        \ifhflip
            \pgftransformyscale{-1}
        \fi
        \pgftransformscale{0.5}
        \pgfpathmoveto{\pgfpoint{-0.5*\overallwidth}{0pt}}
        \pgfpathlineto{\pgfpoint{0.5*\overallwidth}{0pt}}
        \pgfpathmoveto{\pgfpoint{-0.33*\overallwidth}{0.33*\stateheight}}
        \pgfpathlineto{\pgfpoint{0.33*\overallwidth}{0.33*\stateheight}}
        \pgfpathmoveto{\pgfpoint{-0.16*\overallwidth}{0.66*\stateheight}}
        \pgfpathlineto{\pgfpoint{0.16*\overallwidth}{0.66*\stateheight}}
        \pgfpathmoveto{\pgfpoint{-0.02*\overallwidth}{\stateheight}}
        \pgfpathlineto{\pgfpoint{0.02*\overallwidth}{\stateheight}}
        \pgfusepath{stroke}
        \end{pgfonlayer}
    }
}

\usepackage{tikz,xypic}
\usetikzlibrary{decorations.pathreplacing,decorations.markings,arrows.meta,backgrounds,shapes}
\usetikzlibrary{circuits.ee.IEC}
\pgfdeclarelayer{edgelayer}
\pgfdeclarelayer{nodelayer}
\pgfsetlayers{background,edgelayer,nodelayer,main}
\tikzstyle{none}=[inner sep=0mm]
\tikzstyle{every loop}=[]
\tikzstyle{mark coordinate}=[inner sep=0pt,outer sep=0pt,minimum size=3pt,fill=black,circle]

\tikzset{arrow/.style={decoration={
    markings,
    mark=at position #1 with \arrow{>[length=2pt, width=3pt]}},
    postaction=decorate},
    reverse arrow/.style={decoration={
    markings,
    mark=at position #1 with {{\arrow{<[length=2pt, width=3pt]}}}},
    postaction=decorate}
}

\tikzstyle{upground}=[circuit ee IEC,thick,ground,rotate=90,scale=1.5]
\tikzstyle{upgroundwhite}=[circuit ee IEC,thick,ground,rotate=90,scale=1.5, fill=white]
\tikzstyle{downground}=[circuit ee IEC,thick,ground,rotate=-90,scale=1.5]
\tikzstyle{downgroundnorm}=[circuit ee IEC,thick,ground,rotate=-90,scale=1.5, fill=white]

\newcommand{\mapminh}{5mm} 
\newcommand{\stateminh}{5mm}
\newcommand{\maplw}{0.7pt} 
\newcommand{\stateshift}{-0.2pt}
\newcommand{\effectshift}{-0.2pt}

\tikzstyle{box}=[map]
\tikzstyle{medium box}=[medium map]
\tikzstyle{dot}=[inner sep=0mm,minimum width=2mm,minimum height=2mm,draw,shape=circle]  
\tikzstyle{black dot}=[dot,fill=black]
\tikzstyle{white dot}=[dot,fill=white,,text depth=-0.2mm]
\tikzstyle{grey dot}=[dot,fill=black!25] 

\tikzstyle{corner1}=[box,fill=white, font=\footnotesize] %
\tikzstyle{corner2}=[dot,fill=white, font=\footnotesize] %
\tikzstyle{corner3}=[dot,fill=black!25, font=\footnotesize] %
\tikzstyle{corner4}=[dot,fill=black, font=\footnotesize] %

\tikzstyle{scalar}=[circle,draw,inner sep=2pt, line width=\maplw] 

\usetikzlibrary{shapes.misc, positioning}

\tikzset{stateshape/.style={append after command={
   \pgfextra
        \draw[sharp corners, fill=white, line width = \maplw]%
    (\tikzlastnode.west)%
    [rounded corners=0pt] |- (\tikzlastnode.north)%
    [rounded corners=0pt] -| (\tikzlastnode.east)%
    [rounded corners=5pt] |- (\tikzlastnode.south)%
    [rounded corners=5pt] -| (\tikzlastnode.west);
   \endpgfextra}}}

\tikzset{effectshape/.style={append after command={
   \pgfextra
        \draw[sharp corners, fill=white, line width = \maplw]%
    (\tikzlastnode.west)%
    [rounded corners=0pt] |- (\tikzlastnode.south)%
    [rounded corners=0pt] -| (\tikzlastnode.east)%
    [rounded corners=5pt] |- (\tikzlastnode.north)%
    [rounded corners=5pt] -| (\tikzlastnode.west);
   \endpgfextra}}}

 \tikzstyle{map}=[draw,shape=rectangle, inner sep=2pt,minimum height=\mapminh, minimum width=5mm,fill=white]

\tikzstyle{point}=[stateshape,inner sep=2pt, minimum width=6mm, minimum height=\stateminh, yshift=\stateshift]
\tikzstyle{copoint}=[effectshape,inner sep=.2pt, minimum width=6mm, minimum height=\stateminh, yshift=-\effectshift]

\tikzstyle{wide point}=[point, minimum width=12mm]
\tikzstyle{wide copoint}=[copoint, minimum width=12mm]

\tikzstyle{decomp}=[fill=white,draw,shape=isosceles triangle,shape border rotate=-90,isosceles triangle stretches=true,inner sep=0pt,minimum width=0.75cm,minimum height=4mm,yshift=-0.0mm]

\tikzstyle{decompwide}=[fill=white,draw,shape=isosceles triangle,shape border rotate=-90,isosceles triangle stretches=true,inner sep=0pt,minimum width=1.5cm,minimum height=4mm,yshift=-0.0mm]

\tikzstyle{decompflip}=[fill=white,draw,shape=isosceles triangle,shape border rotate=90,isosceles triangle stretches=true,inner sep=0pt,minimum width=0.75cm,minimum height=4mm,yshift=-0.0mm]

\tikzstyle{decompwideflip}=[fill=white,draw,shape=isosceles triangle,shape border rotate=90,isosceles triangle stretches=true,inner sep=0pt,minimum width=1.5cm,minimum height=4mm,yshift=-0.0mm]

\tikzstyle{medium map} = [map, minimum width = 12mm] 
\tikzstyle{semilarge map} = [map, minimum width = 15mm] 
\tikzstyle{large map} = [map, minimum width = 18mm] 

\tikzstyle{kpoint} =[point]
\tikzstyle{kpointadj} =[copoint]
\tikzstyle{kpointconj}=[dagpointconj] 

\makeatletter
\newcommand{\boxshape}[3]{%
\pgfdeclareshape{#1}{
\inheritsavedanchors[from=rectangle] 
\inheritanchorborder[from=rectangle]
\inheritanchor[from=rectangle]{center}
\inheritanchor[from=rectangle]{north}
\inheritanchor[from=rectangle]{south}
\inheritanchor[from=rectangle]{west}
\inheritanchor[from=rectangle]{east}
\backgroundpath{
\southwest \pgf@xa=\pgf@x \pgf@ya=\pgf@y
\northeast \pgf@xb=\pgf@x \pgf@yb=\pgf@y

\@tempdima=#2
\@tempdimb=#3

\pgfpathmoveto{\pgfpoint{\pgf@xa - 5pt + \@tempdima}{\pgf@ya}}
\pgfpathlineto{\pgfpoint{\pgf@xa - 5pt - \@tempdima}{\pgf@yb}}
\pgfpathlineto{\pgfpoint{\pgf@xb + 5pt + \@tempdimb}{\pgf@yb}}
\pgfpathlineto{\pgfpoint{\pgf@xb + 5pt - \@tempdimb}{\pgf@ya}}
\pgfpathlineto{\pgfpoint{\pgf@xa - 5pt + \@tempdima}{\pgf@ya}}
\pgfpathclose
}
}}

\boxshape{NEbox}{0pt}{3pt} 
\boxshape{SEbox}{0pt}{-3pt}
\boxshape{NWbox}{3pt}{0pt}
\boxshape{SWbox}{-3pt}{0pt}
\boxshape{rec-box}{0pt}{0pt}
\makeatother

\tikzstyle{cloud}=[shape=cloud,draw,minimum width=1.5cm,minimum height=1.5cm]

\tikzstyle{dagmap}=[draw,shape=NEbox,inner sep=2pt,minimum height=\mapminh,fill=white, line width = \maplw] %
\tikzstyle{dashedmap}=[draw,dashed,shape=NEbox,inner sep=2pt,minimum height=\mapminh,fill=white, line width = \maplw]
\tikzstyle{mapdag}=[draw,shape=SEbox,inner sep=2pt,minimum height=\mapminh,fill=white, line width = \maplw]
\tikzstyle{mapadj}=[draw,shape=SEbox,inner sep=2pt,minimum height=\mapminh,fill=white, line width = \maplw]
\tikzstyle{maptrans}=[draw,shape=SWbox,inner sep=2pt,minimum height=\mapminh,fill=white, line width = \maplw]
\tikzstyle{mapconj}=[draw,shape=NWbox,inner sep=2pt,minimum height=\mapminh,fill=white, line width = \maplw]

\tikzstyle{medium dagmap}=[draw,shape=NEbox,inner sep=2pt,minimum height=\mapminh,fill=white,minimum width=7mm, line width = \maplw]
\tikzstyle{semilarge dagmap}=[draw,shape=NEbox,inner sep=2pt,minimum height=\mapminh,fill=white,minimum width=9.5mm, line width = \maplw]
\tikzstyle{large dagmap}=[draw,shape=NEbox,inner sep=2pt,minimum height=\mapminh,fill=white,minimum width=12mm, line width = \maplw]

\makeatletter

\pgfdeclareshape{cornerpoint}{
\inheritsavedanchors[from=rectangle] 
\inheritanchorborder[from=rectangle]
\inheritanchor[from=rectangle]{center}
\inheritanchor[from=rectangle]{north}
\inheritanchor[from=rectangle]{south}
\inheritanchor[from=rectangle]{west}
\inheritanchor[from=rectangle]{east}
\backgroundpath{
\southwest \pgf@xa=\pgf@x \pgf@ya=\pgf@y
\northeast \pgf@xb=\pgf@x \pgf@yb=\pgf@y

\pgfmathsetmacro{\pgf@shorten@left}{\pgfkeysvalueof{/tikz/shorten left}}
\pgfmathsetmacro{\pgf@shorten@right}{\pgfkeysvalueof{/tikz/shorten right}}

\pgfpathmoveto{\pgfpoint{0.5 * (\pgf@xa + \pgf@xb)}{\pgf@ya - 5pt}}
\pgfpathlineto{\pgfpoint{\pgf@xa - 8pt + \pgf@shorten@left}{\pgf@yb - 1.5 * \pgf@shorten@left}}
\pgfpathlineto{\pgfpoint{\pgf@xa - 8pt + \pgf@shorten@left}{\pgf@yb}}
\pgfpathlineto{\pgfpoint{\pgf@xb + 8pt - \pgf@shorten@right}{\pgf@yb}}
\pgfpathlineto{\pgfpoint{\pgf@xb + 8pt - \pgf@shorten@right}{\pgf@yb - 1.5 * \pgf@shorten@right}}
\pgfpathclose
}
}

\pgfdeclareshape{cornercopoint}{
\inheritsavedanchors[from=rectangle] 
\inheritanchorborder[from=rectangle]
\inheritanchor[from=rectangle]{center}
\inheritanchor[from=rectangle]{north}
\inheritanchor[from=rectangle]{south}
\inheritanchor[from=rectangle]{west}
\inheritanchor[from=rectangle]{east}
\backgroundpath{
\southwest \pgf@xa=\pgf@x \pgf@ya=\pgf@y
\northeast \pgf@xb=\pgf@x \pgf@yb=\pgf@y

\pgfmathsetmacro{\pgf@shorten@left}{\pgfkeysvalueof{/tikz/shorten left}}
\pgfmathsetmacro{\pgf@shorten@right}{\pgfkeysvalueof{/tikz/shorten right}}

\pgfpathmoveto{\pgfpoint{0.5 * (\pgf@xa + \pgf@xb)}{\pgf@yb + 5pt}}
\pgfpathlineto{\pgfpoint{\pgf@xa - 8pt + \pgf@shorten@left}{\pgf@ya + 1.5 * \pgf@shorten@left}}
\pgfpathlineto{\pgfpoint{\pgf@xa - 8pt + \pgf@shorten@left}{\pgf@ya}}
\pgfpathlineto{\pgfpoint{\pgf@xb + 8pt - \pgf@shorten@right}{\pgf@ya}}
\pgfpathlineto{\pgfpoint{\pgf@xb + 8pt - \pgf@shorten@right}{\pgf@ya + 1.5 * \pgf@shorten@right}}
\pgfpathclose
}
}

\makeatother

\pgfkeyssetvalue{/tikz/shorten left}{0pt}
\pgfkeyssetvalue{/tikz/shorten right}{0pt}

\tikzstyle{dagpoint common}=[draw,fill=white,inner sep=1pt, line width = \maplw, minimum height = 4mm, yshift=1.2pt] 
\tikzstyle{dagpoint sc}=[shape=cornerpoint,dagpoint common]
\tikzstyle{dagpoint adjoint sc}=[shape=cornercopoint,dagpoint common]
\tikzstyle{dagpoint}=[shape=cornerpoint,shorten left=4pt,dagpoint common]
\tikzstyle{dagpointadj}=[shape=cornercopoint,shorten left=5pt,dagpoint common]
\tikzstyle{dagpointconj}=[shape=cornerpoint,shorten right=5pt,dagpoint common]
\tikzstyle{dagpointtrans}=[shape=cornercopoint,shorten right=5pt,dagpoint common]
\tikzstyle{dagpointsymm}=[shape=cornerpoint,shorten left=5pt,shorten right=5pt,dagpoint common]

\tikzstyle{widedagpoint}=[dagpoint, minimum width=1 cm, inner sep=2pt]
\tikzstyle{widedagpointadj}=[dagpointadj, minimum width=1 cm, inner sep=2pt]

\tikzstyle{every picture}=[baseline=-0.25em,scale=0.5]
\tikzstyle{label}=[font=\footnotesize,text height=1ex, text depth=0.15ex]

\usetikzlibrary[shapes]

\tikzset{
sidetriangle/.style = {regular polygon, regular polygon sides = 3, aspect = 1, shape border rotate = 90, draw, inner sep = 0, minimum width = 1.2cm}
}

\tikzset{
isoc/.style = {shape=isosceles triangle, shape border rotate = 180, isosceles triangle stretches = true, minimum width = 1.2cm, minimum height= 1.5cm, inner sep = 0.3}}

\tikzset{
coarse/.style = {shape = circle, fill = white, draw, inner sep = 0, minimum width =0.125cm}
}
\tikzset{
coarsesymbol/.style = {shape = circle, fill = white, inner sep = -0.7, minimum width = 0.125cm}
}

\tikzstyle{sidetriangle2}=[sidetriangle, minimum width = 2cm, fill=white]
\tikzstyle{sideisocsmall}]=[style=isoc, minimum width = 1cm, minimum height = 0.8cm, draw, fill=white, font=\Large]
\tikzstyle{sideisoc}]=[style=isoc, minimum width = 2cm, draw, fill=white, font=\Large]
\tikzstyle{sideisocmid}]=[style=isoc, minimum width = 2.5cm, draw, fill=white, font=\Large]
\tikzstyle{sideisocmedium}]=[style=isoc, minimum width = 3cm, draw, fill=white, font=\Large]

\newcommand{\tinygroundflipnew}{
\smash{
{\hspace{-3pt}
\ensuremath{
\begin{picc}[yscale=-1.0] 
    \node[upground, xscale=0.8, yscale=-0.7] (1) at (0,0.10) {};
    \draw (0,-0.03) to (0,-0.31);
\end{picc}
}\hspace{-1pt}}}}

\tikzstyle{label}=[font=\footnotesize,text height=1ex, text depth=0.15ex]

\tikzstyle{box}=[map]
\tikzstyle{medium box}=[medium map]
\tikzstyle{dot}=[inner sep=0mm,minimum width=2mm,minimum height=2mm,draw,shape=circle]  
\tikzstyle{black dot}=[dot,fill=black]
\tikzstyle{white dot}=[dot,fill=white,,text depth=-0.2mm]
\tikzstyle{grey dot}=[dot,fill=black!25] 

\tikzstyle{corner1}=[box,fill=white, font=\footnotesize] %
\tikzstyle{corner2}=[dot,fill=white, font=\footnotesize] %
\tikzstyle{corner3}=[dot,fill=black!25, font=\footnotesize] %
\tikzstyle{corner4}=[dot,fill=black, font=\footnotesize] %

\tikzstyle{scalar}=[circle,draw,inner sep=2pt, line width=\maplw] 

\usetikzlibrary{shapes.misc, positioning}

\tikzset{stateshape/.style={append after command={
   \pgfextra
        \draw[sharp corners, fill=white, line width = \maplw]%
    (\tikzlastnode.west)%
    [rounded corners=0pt] |- (\tikzlastnode.north)%
    [rounded corners=0pt] -| (\tikzlastnode.east)%
    [rounded corners=5pt] |- (\tikzlastnode.south)%
    [rounded corners=5pt] -| (\tikzlastnode.west);
   \endpgfextra}}}

\tikzset{effectshape/.style={append after command={
   \pgfextra
        \draw[sharp corners, fill=white, line width = \maplw]%
    (\tikzlastnode.west)%
    [rounded corners=0pt] |- (\tikzlastnode.south)%
    [rounded corners=0pt] -| (\tikzlastnode.east)%
    [rounded corners=5pt] |- (\tikzlastnode.north)%
    [rounded corners=5pt] -| (\tikzlastnode.west);
   \endpgfextra}}}

\tikzstyle{point}=[stateshape,inner sep=2pt, minimum width=6mm, minimum height=\stateminh, yshift=\stateshift]
\tikzstyle{copoint}=[effectshape,inner sep=.2pt, minimum width=6mm, minimum height=\stateminh, yshift=-\effectshift]

\tikzstyle{wide point}=[point, minimum width=12mm]
\tikzstyle{wide copoint}=[copoint, minimum width=12mm]

%% file: figures/effect.tikz
\begin{tikzpicture}
	\begin{pgfonlayer}{nodelayer}
		\node [style=map] (0) at (0, 0.75) {$C$};
		\node [style=none] (1) at (0, -0.25) {};
		\node [style=label] (2) at (0, -0.75) {$X$};
	\end{pgfonlayer}
	\begin{pgfonlayer}{edgelayer}
		\draw (1.center) to (0);
	\end{pgfonlayer}
\end{tikzpicture}

%% file: figures/state.tikz
\begin{tikzpicture}
	\begin{pgfonlayer}{nodelayer}
		\node [style=map] (0) at (0, -0.25) {$\omega$};
		\node [style=none] (1) at (0, 0.75) {};
		\node [style=label] (2) at (0, 1) {$X$};
	\end{pgfonlayer}
	\begin{pgfonlayer}{edgelayer}
		\draw (1.center) to (0);
	\end{pgfonlayer}
\end{tikzpicture}

%% file: figures/state-effect.tikz
\begin{tikzpicture}
	\begin{pgfonlayer}{nodelayer}
		\node [style=map] (0) at (0, -0.75) {$\omega$};
		\node [style=map] (1) at (0, 1.25) {$C$};
		\node [style=label] (2) at (-0.5, 0.25) {$X$};
	\end{pgfonlayer}
	\begin{pgfonlayer}{edgelayer}
		\draw (1) to (0);
	\end{pgfonlayer}
\end{tikzpicture}

%% file: figures/copy-delete.tikz
\begin{tikzpicture}
	\begin{pgfonlayer}{nodelayer}
		\node [style=whitedot] (0) at (0, -0.25) {};
		\node [style=none] (1) at (0, -1) {};
		\node [style=none] (2) at (-0.75, 0.5) {};
		\node [style=none] (5) at (0.75, 0.5) {};
		\node [style=none] (6) at (4.5, 0) {};
		\node [style=none] (7) at (4.5, -1) {};
		\node [style=upground] (8) at (4.5, 0.25) {};
	\end{pgfonlayer}
	\begin{pgfonlayer}{edgelayer}
		\draw (1.center) to (0);
		\draw [bend left=45] (0) to (2.center);
		\draw [bend right=45] (0) to (5.center);
		\draw (7.center) to (6.center);
	\end{pgfonlayer}
\end{tikzpicture}

%% file: figures/update.tikz
\begin{tikzpicture}
	\begin{pgfonlayer}{nodelayer}
		\node [style=whitedot] (0) at (0, -0.5) {};
		\node [style=none] (1) at (0, -1.5) {};
		\node [style=none] (2) at (0.75, 0.5) {};
		\node [style=none] (5) at (-0.75, 0.5) {};
		\node [style=map] (7) at (-0.75, 0.5) {$C$};
		\node [style=none] (8) at (0.75, 1) {};
		\node [style=none] (9) at (-0.75, 0.5) {};
	\end{pgfonlayer}
	\begin{pgfonlayer}{edgelayer}
		\draw (1.center) to (0);
		\draw [bend right=45] (0) to (2.center);
		\draw [bend left=45] (0) to (5.center);
		\draw (9.center) to (7);
		\draw (8.center) to (2.center);
	\end{pgfonlayer}
\end{tikzpicture}

%% file: figures/CD.tikz
\begin{tikzpicture}
	\begin{pgfonlayer}{nodelayer}
		\node [style=whitedot] (0) at (0, -0.5) {};
		\node [style=none] (1) at (0, -1.5) {};
		\node [style=none] (2) at (-0.75, 0.5) {};
		\node [style=none] (5) at (0.75, 0.5) {};
		\node [style=map] (7) at (-0.75, 0.5) {$C$};
		\node [style=none] (8) at (-0.75, 0.5) {};
		\node [style=map] (9) at (0.75, 0.5) {$D$};
	\end{pgfonlayer}
	\begin{pgfonlayer}{edgelayer}
		\draw (1.center) to (0);
		\draw [bend left=45] (0) to (2.center);
		\draw [bend right=45] (0) to (5.center);
		\draw (8.center) to (2.center);
	\end{pgfonlayer}
\end{tikzpicture}

%% file: figures/crisp.tikz
\begin{tikzpicture}
	\begin{pgfonlayer}{nodelayer}
		\node [style=whitedot] (0) at (0, -0.5) {};
		\node [style=none] (1) at (0, -1.5) {};
		\node [style=none] (2) at (-0.75, 0.5) {};
		\node [style=none] (5) at (0.75, 0.5) {};
		\node [style=map] (6) at (-0.75, 0.5) {$f$};
		\node [style=map] (7) at (0.75, 0.5) {$f$};
		\node [style=none] (8) at (-0.75, 1.5) {};
		\node [style=none] (9) at (0.75, 1.5) {};
		\node [style=none] (10) at (2.75, 0) {$=$};
		\node [style=whitedot] (11) at (5.25, 0.5) {};
		\node [style=none] (12) at (5.25, -1.5) {};
		\node [style=none] (13) at (4.5, 1.5) {};
		\node [style=none] (14) at (6, 1.5) {};
		\node [style=none] (15) at (4.5, 1.5) {};
		\node [style=none] (16) at (6, 1.5) {};
		\node [style=none] (17) at (4.5, 1.5) {};
		\node [style=none] (18) at (6, 1.5) {};
		\node [style=map] (19) at (5.25, -0.5) {$f$};
	\end{pgfonlayer}
	\begin{pgfonlayer}{edgelayer}
		\draw (1.center) to (0);
		\draw [bend left=45] (0) to (2.center);
		\draw [bend right=45] (0) to (5.center);
		\draw (8.center) to (6);
		\draw (9.center) to (7);
		\draw (12.center) to (11);
		\draw [bend left=45] (11) to (13.center);
		\draw [bend right=45] (11) to (14.center);
		\draw (17.center) to (15.center);
		\draw (18.center) to (16.center);
	\end{pgfonlayer}
\end{tikzpicture}

%% file: figures/convolution.tikz
\begin{tikzpicture}
	\begin{pgfonlayer}{nodelayer}
		\node [style=whitedot] (0) at (0, 0.25) {$ \! + \!$};
		\node [style=none] (1) at (0, 1) {};
		\node [style=none] (2) at (-0.75, -1) {};
		\node [style=none] (5) at (0.75, -1) {};
		\node [style=map] (6) at (-0.75, -1) {$f$};
		\node [style=map] (7) at (0.75, -1) {$g$};
		\node [style=whitedot] (8) at (0, -2.25) {};
		\node [style=none] (9) at (0, -3) {};
		\node [style=none] (10) at (-0.75, -1) {};
		\node [style=none] (11) at (0.75, -1) {};
	\end{pgfonlayer}
	\begin{pgfonlayer}{edgelayer}
		\draw (1.center) to (0);
		\draw [bend right=45] (0) to (2.center);
		\draw [bend left=45] (0) to (5.center);
		\draw (9.center) to (8);
		\draw [bend left=45] (8) to (10.center);
		\draw [bend right=45] (8) to (11.center);
	\end{pgfonlayer}
\end{tikzpicture}

%% file: figures/add-noise.tikz
\begin{tikzpicture}
	\begin{pgfonlayer}{nodelayer}
		\node [style=whitedot] (0) at (0, 0.25) {$ \! + \!$};
		\node [style=none] (1) at (0, 1) {};
		\node [style=none] (2) at (-0.75, -1) {};
		\node [style=none] (5) at (0.75, -1) {};
		\node [style=map] (6) at (-0.75, -1) {$f$};
		\node [style=map] (7) at (0.75, -1) {$\nu$};
		\node [style=none] (9) at (-0.75, -2) {};
		\node [style=none] (10) at (-0.75, -1) {};
	\end{pgfonlayer}
	\begin{pgfonlayer}{edgelayer}
		\draw (1.center) to (0);
		\draw [bend right=45] (0) to (2.center);
		\draw [bend left=45] (0) to (5.center);
		\draw (9.center) to (10.center);
	\end{pgfonlayer}
\end{tikzpicture}

%% file: figures/green-extended.tikz
\begin{tikzpicture}
	\begin{pgfonlayer}{nodelayer}
		\node [style=map] (0) at (-0.75, 0.5) {Green};
		\node [style=none] (1) at (-0.75, -0.5) {};
		\node [style=label] (2) at (-0.75, -1) {$F$};
		\node [style=none] (3) at (1, 0) {$:=$};
		\node [style=map] (4) at (3, 0.5) {Green};
		\node [style=none] (5) at (3, -0.5) {};
		\node [style=label] (6) at (3, -1) {$C$};
		\node [style=none] (7) at (4.75, -0.5) {};
		\node [style=upground] (8) at (4.75, 0.25) {};
		\node [style=label] (9) at (4.75, -1) {$T$};
	\end{pgfonlayer}
	\begin{pgfonlayer}{edgelayer}
		\draw (0) to (1.center);
		\draw (4) to (5.center);
		\draw (8) to (7.center);
	\end{pgfonlayer}
\end{tikzpicture}

%% file: figures/banana.tikz
\begin{tikzpicture}
	\begin{pgfonlayer}{nodelayer}
		\node [style=map] (1) at (-2.25, 0.5) {Banana};
		\node [style=none] (2) at (-2.25, -0.5) {};
		\node [style=label] (3) at (-2.25, -1) {$F$};
		\node [style=none] (4) at (-0.25, 0) {$=$};
		\node [style=map] (5) at (1.8, 0.5) {Yellow};
		\node [style=none] (6) at (1.8, -0.5) {};
		\node [style=label] (7) at (1.8, -1) {$C$};
		\node [style=map] (8) at (4.25, 0.5) {Sweet};
		\node [style=none] (9) at (4.25, -0.5) {};
		\node [style=label] (10) at (4.25, -1) {$T$};
		\node [style=none] (11) at (6.25, 0) {$\bigvee$};
		\node [style=map] (12) at (8.25, 0.5) {Green};
		\node [style=none] (13) at (8.25, -0.5) {};
		\node [style=label] (14) at (8.25, -1) {$C$};
		\node [style=map] (15) at (10.5, 0.5) {Bitter};
		\node [style=none] (16) at (10.5, -0.5) {};
		\node [style=label] (17) at (10.5, -1) {$T$};
	\end{pgfonlayer}
	\begin{pgfonlayer}{edgelayer}
		\draw (1) to (2.center);
		\draw (5) to (6.center);
		\draw (8) to (9.center);
		\draw (12) to (13.center);
		\draw (15) to (16.center);
	\end{pgfonlayer}
\end{tikzpicture}

%% file: figures/banana-fuzzy.tikz
\begin{tikzpicture}
	\begin{pgfonlayer}{nodelayer}
		\node [style=map] (0) at (0.5, 0.5) {banana};
		\node [style=none] (1) at (0.5, -0.5) {};
		\node [style=label] (2) at (0.5, -1) {$F$};
	\end{pgfonlayer}
	\begin{pgfonlayer}{edgelayer}
		\draw (0) to (1.center);
	\end{pgfonlayer}
\end{tikzpicture}

%% file: figures/greenbanana.tikz
\begin{tikzpicture}
	\begin{pgfonlayer}{nodelayer}
		\node [style=whitedot] (0) at (2.25, -0.25) {};
		\node [style=none] (1) at (2.25, -1.25) {};
		\node [style=none] (2) at (1, 0.75) {};
		\node [style=none] (5) at (3.5, 0.75) {};
		\node [style=map] (7) at (1, 0.75) {green};
		\node [style=map] (9) at (3.5, 0.75) {banana};
		\node [style=map] (10) at (-3.75, 0.5) {green banana};
		\node [style=none] (11) at (-3.75, -1.25) {};
		\node [style=none] (12) at (-1, -0.25) {$=$};
		\node [style=label] (13) at (-3.75, -1.75) {$F$};
		\node [style=label] (14) at (2.25, -1.75) {$F$};
	\end{pgfonlayer}
	\begin{pgfonlayer}{edgelayer}
		\draw (1.center) to (0);
		\draw [bend left=45] (0) to (2.center);
		\draw [bend right=45] (0) to (5.center);
		\draw (11.center) to (10);
	\end{pgfonlayer}
\end{tikzpicture}

%% file: figures/metaphor.tikz
\begin{tikzpicture}
	\begin{pgfonlayer}{nodelayer}
		\node [style=whitedot] (0) at (0, -0.5) {};
		\node [style=downground] (1) at (0, -1.25) {};
		\node [style=none] (5) at (1, 0.5) {};
		\node [style=medium map] (9) at (1.5, 0.75) {banana};
		\node [style=none] (10) at (2, 0.5) {};
		\node [style=none] (11) at (2, -1.5) {};
		\node [style=label] (12) at (2, -2) {$T$};
		\node [style=label] (13) at (0, -2) {$C$};
		\node [style=none] (14) at (0, -0.5) {};
		\node [style=none] (15) at (-1, 0.5) {};
		\node [style=none] (16) at (-1, 0.75) {};
		\node [style=label] (17) at (-1, 1.25) {$C$};
		\node [style=map] (18) at (-4, -0.5) {`tastes'};
		\node [style=none] (19) at (-4, -1.5) {};
		\node [style=none] (20) at (-4, 0.75) {};
		\node [style=none] (21) at (-2, -0.5) {$:=$};
		\node [style=label] (22) at (-4, 1.25) {$C$};
		\node [style=label] (23) at (-4, -2) {$T$};
	\end{pgfonlayer}
	\begin{pgfonlayer}{edgelayer}
		\draw (1) to (0);
		\draw [bend right=45] (0) to (5.center);
		\draw (11.center) to (10.center);
		\draw [bend left=45] (14.center) to (15.center);
		\draw (16.center) to (15.center);
		\draw (20.center) to (19.center);
	\end{pgfonlayer}
\end{tikzpicture}

%% file: figures/tastes-yellow-reduced.tikz
\begin{tikzpicture}
	\begin{pgfonlayer}{nodelayer}
		\node [style=map] (18) at (-4.75, -0.75) {`tastes'};
		\node [style=none] (19) at (-4.75, -1.75) {};
		\node [style=none] (20) at (-4.75, 0.5) {};
		\node [style=label] (23) at (-4.75, -2.25) {$T$};
		\node [style=map] (28) at (-4.75, 0.5) {yellow};
		\node [style=map] (31) at (-10.25, 0) {`tastes yellow'};
		\node [style=none] (32) at (-10.25, -1.75) {};
		\node [style=none] (33) at (-10.25, 0) {};
		\node [style=label] (34) at (-10.25, -2.25) {$T$};
		\node [style=none] (35) at (-7.25, -0.5) {$=$};
	\end{pgfonlayer}
	\begin{pgfonlayer}{edgelayer}
		\draw (20.center) to (19.center);
		\draw (33.center) to (32.center);
	\end{pgfonlayer}
\end{tikzpicture}